\documentclass[11pt]{article}
\usepackage{amssymb,latexsym,amsmath,amsbsy,amsthm,amsxtra,amsgen,graphicx,dsfont,color,authblk}
\newfont{\msbm}{msbm10 at 11pt}

\theoremstyle{plain}
\newtheorem{theorem}{Theorem}[section]                                          
\newtheorem{proposition}[theorem]{Proposition}                          
\newtheorem{lemma}[theorem]{Lemma}
\newtheorem{corollary}[theorem]{Corollary}

\theoremstyle{definition}

\theoremstyle{remark}
\newtheorem{remark}[theorem]{Remark}

\def\XbT{X_{1,I_1}}

\def\barYbT{\bar{Y}_{1,I_1}}
\def\YbT{Y_{1,I_1}}
\def\tilYbT{\tilde{Y}_{1,I_1}}

\def\tilMYsur{\tilde{M}_{1,s}}

\def\taukmut{\tau_{k}}
\def\XbTk{X_{k,I_k}}
\def\YbTk{Y_{k,I_k}}
\def\barYbTk{\bar{Y}_{k,I_k}}
\def\tilYbTk{\tilde{Y}_{k,I_k}}

\def\XtoaN{X_1^{(1)}}

\begin{document}

\title{The Accumulation of Beneficial Mutations and Convergence to a Poisson Process}
\author[1]{Nantawat Udomchatpitak\thanks{This work (Grant No. RGNS 64-155) was supported by Office of the Permanent Secretary, Ministry of Higher Education, Science, Research and Innovation  (OPS MHESI), Thailand Science Research and Innovation (TSRI) and Mahidol University.}}
\author[2]{Jason Schweinsberg}
\affil[1]{\fontsize{9}{10.8}\itshape{Department of Mathematics, Mahidol University}}
\affil[2]{\fontsize{9}{10.8}\itshape{Department of Mathematics, University of California San Diego}}
\date{}
\maketitle

\begin{abstract}
We consider a model of a population with fixed size $N$, which is subjected to an unlimited supply of beneficial mutations at a constant rate $\mu_N$.  Individuals with $k$ beneficial mutations have the fitness $(1+s_N)^k$. Each individual dies at rate 1 and is replaced by a random individual chosen with probability proportional to its fitness.  We show that when $\mu_N \ll 1/(N \log N)$ and $N^{-\eta} \ll s_N \ll 1$ for some $\eta < 1$, the fixation times of beneficial mutations, after a time scaling, converge to the times of a Poisson process, even though for some choices of $s_N$ and $\mu_N$ satisfying these conditions, there will sometimes be multiple beneficial mutations with distinct origins in the population, competing against each other.

\end{abstract}

{\small MSC: Primary 92D15; Secondary 60J27, 60J80, 92D25

}

\section{Introduction}

One of the most important questions in evolutionary biology is to understand how beneficial mutations accumulate in a population.  We consider here a simple model of a population which repeatedly acquires beneficial mutations.  We assume the population has fixed size $N$.  We assume that, at time zero, no individuals have mutations, but then each individual in the population independently acquires mutations at times of a homogeneous Poisson process with rate $\mu_N$.  All mutations are assumed to be beneficial and to increase the individual's fitness by a factor of $1 + s_N$, so that an individual with $k$ mutations, which we call a type $k$ individual, has fitness $(1 + s_N)^k$.  We assume that each individual independently lives for an exponentially distributed time with rate $1$.  When an individual dies, it gets replaced by a new individual whose parent is chosen at random from the $N$ individuals in the population, with probability proportional to the individual's fitness.  The new individual inherits all of its parent's mutations.

It is instructive to consider what happens after one individual acquires a beneficial mutation, if we assume that no further mutations can occur.  As we will explain in more detail below, the number of individuals with the mutation then evolves like a birth and death chain in which the ratio of the birth rate to the death rate is $1 + s_N$.  Classical results on asymmetric random walks imply that the probability that this chain reaches $N$ before $0$ is 
\begin{equation*}
	\frac{s_N}{(1 + s_N)(1 - (1 + s_N)^{-N})},
\end{equation*} 
which is approximately $s_N/(1+ s_N)$ as long as $(1 + s_N)^N \rightarrow \infty$ as $N \rightarrow \infty$.  Therefore, the beneficial mutation may quickly disappear, but with probability approximately $s_N/(1 + s_N)$, the beneficial mutation will spread to the entire population, an event known as a selective sweep.  One can also show that the duration of a selective sweep, that is, the time required for a beneficial mutation to spread to the entire population, is approximately $(2/s_N) \log N$.  This question was first investigated by Kimura and Ohta \cite{ko69}, and a rigorous analysis for a population model very similar to the one presented here is given in section 6.1 of \cite{durr08}.

Returning now to original population model, because there are $N$ individuals acquiring mutations at rate $\mu_N$, the total mutation rate for the population is $N \mu_N$.  Therefore, the rate of mutations which trigger a selective sweep is approximately $N \mu_N s_N/(1 + s_N)$.  It follows that the expected time between such mutations is approximately $(1 + s_N)/(N \mu_N s_N)$.  Therefore, the time between selective sweeps is much longer than the duration of a selective sweep provided that $\mu_N \ll 1/(N \log N)$.  As a result, when $\mu_N \ll 1/(N \log N)$, we expect to have approximately exponentially distributed waiting times between selective sweeps, so that after a suitable rescaling of time, the times of selective sweeps converge as $N \rightarrow \infty$ to the times of a homogeneous Poisson process.

When $s_N$ is bounded away from zero as $N \rightarrow \infty$, which is the case of strong selection, this result is straightforward to prove because with high probability, there will only be one beneficial mutation in the population at any given time that has not already spread to the entire population.  As a result, the selective sweeps can be analyzed individually.  However, when $s_N \rightarrow 0$ as $N \rightarrow \infty$, the analysis can become more complicated.  To understand why, suppose that at some time, all individuals have type zero, and then a beneficial mutation occurs which triggers a selective sweep of duration approximately $(2/s_N) \log N$.  Then the expected number of additional mutations which arise during the sweep will be approximately $(2/s_N) N \mu_N \log N$, which will tend to infinity if $s_N/(N \log N) \ll \mu_N \ll 1/(N \log N)$.  Some of these new mutations will cause type 0 individuals to become type 1, making them indistinguishable in the model from descendants of the original mutation.  Other mutations will cause type 1 individuals to become type 2, and then depending on the choice of parameters, these type 2 individuals could pick up an additional mutation and become type 3, and so on.  The presence in the population of multiple beneficial mutations with distinct origins leads to what is known as clonal interference and complicates the analysis significantly.  Nevertheless, because these additional mutations ultimately die out before spreading to a large number of individuals, one can prove that selective sweeps occur approximately at times of a Poisson process for a range of values of $s_N$ which includes the case of moderate selection, where $s_N = N^{-b}$ for some $b \in (0,1)$.

Given two sequences of positive numbers $(a_N)_{N=1}^{\infty}$ and $(b_N)_{N=1}^{\infty}$, we write $a_N \ll b_N$ if $\lim_{N \rightarrow \infty} a_N/b_N = 0$ and $a_N \sim b_N$ if $\lim_{N \rightarrow \infty} a_N/b_N = 1$.  Throughout the paper, we will assume that
\begin{equation}\label{muN}
	0<\mu_N \ll \frac{1}{N \log N}
\end{equation}
and that
\begin{equation}\label{sN}
	N^{-\eta} \ll s_N \ll 1 \quad\textup{for some }\eta < 1.
\end{equation}
Let $X_{k,N}(t)$ be the number of individuals at time $t$ with exactly $k$ mutations, that is, the number of type $k$ individuals.  Let $T_{0,N} = 0$, and let 
\begin{equation*}
	T_{k,N} = \inf\bigg\{t \geq 0: X_{k,N}(t) > \frac{\log N}{s_N} \bigg\}.
\end{equation*}  
Also, let 
\begin{equation*}
	\Delta= \left\lfloor \frac{1}{1 - \eta} \right\rfloor+1.
\end{equation*}
The following theorem is the main result of this paper.

\begin{theorem}\label{mainth}
	Assume that \eqref{muN} and \eqref{sN} hold.  Let $(\xi_k)_{k=1}^{\infty}$ be a sequence of independent random variables having the exponential distribution with mean one.  Then for each fixed positive integer $K$, as $N \rightarrow \infty$ we have the convergence in distribution
	\begin{equation}\label{PPPconvergence}
		\big(N \mu_N s_N(T_{k,N} - T_{k-1,N}) \big)_{k=1}^K \Rightarrow (\xi_k)_{k=1}^K.
	\end{equation}
	Furthermore, there exist positive constants $C_1$ and $C_2$, depending on $\eta$, such that for all nonnegative integers $k$, we have
	\begin{align}\label{fastsweep}
\lim_{N \rightarrow \infty} P\bigg(X_{k,N}(t) \geq N - \frac{C_2 \log N}{s_N} \textup{ for all $t$ such that } T_{k,N} + \frac{C_1 \log N}{s_N} \leq t < T_{k+1, N} \bigg) = 1
	\end{align}
	and
	\begin{align}\label{width}
\lim_{N \rightarrow \infty} P \bigg( \sum_{j=k}^{k+\Delta} X_{j,N}(t) = N \textup{ for all $t$ such that } T_{k,N} + \frac{C_1 \log N}{s_N} \leq t < T_{k+1, N} \bigg) = 1.
	\end{align} 
\end{theorem}

We can think of $T_{k,N}$ as being approximately the time when type $k$ becomes established in the population.  The result \eqref{PPPconvergence} demonstrates that the times $T_{k,N}$, when scaled by $N \mu_N s_N$ which is approximately the rate at which selective sweeps take place, converge as $N \rightarrow \infty$ to the times of a homogeneous rate one Poisson process.  The result \eqref{fastsweep} shows that shortly after time $T_{k,N}$, most of the population consists of type $k$ individuals.  Furthermore, the result \eqref{width} shows that all individuals of types $k-1$ and lower disappear from the population shortly after time $T_{k,N}$, and then all individuals have types between $k$ and $k + \Delta$ until at least time $T_{k+1,N}$.

From this result, we obtain the following corollary regarding how the average number of mutations in the population evolves over time.

\begin{corollary}\label{maincor}
	Assume that \eqref{muN} and \eqref{sN} hold.  For all $t \geq 0$, let 
	\begin{equation*}
		\overline{X}_N(t) = \frac{1}{N} \sum_{k=0}^{\infty} k X_{k,N}(t)
	\end{equation*} 
	denote the average number of mutations carried by the $N$ individuals in the population at time $t$.  Then, as $N \rightarrow \infty$ the finite-dimensional distributions of the processes $(\overline{X}_N(t/(N \mu_N s_N)), t \geq 0)$ converge  to the finite-dimensional distributions of a homogeneous rate one Poisson process.
\end{corollary}

To give one indication of why these results are significant, we refer the reader to the award-winning papers \cite{gkwy16, bgpw19}, which provide a mathematical analysis of the results of the famous Lenski experiments on bacterial evolution.  In \cite{gkwy16}, the authors consider a model very similar to the one in the present paper and assume that $s_N \sim N^{-b}$ and $\mu_N \sim N^{-(1 + a)}$, where $0 < b < 1/2$ and $a > 3b$.  (In \cite{gkwy16}, the mutation rate is written as $\mu_N \sim N^{-a}$, but this is because $\mu_N$ in \cite{gkwy16} refers to the mutation rate for the entire population and therefore corresponds to $N \mu_N$ in the present paper.)  These restrictions on the parameters are chosen to eliminate all clonal interference on the time scale of interest with high probability.  Theorem~\ref{mainth} suggests that the same results may still hold if the condition $a > 3b$ is replaced by the weaker condition $a > 0$, which is sufficient to eliminate clonal interference among beneficial mutations that do not quickly die out.

Many papers have been devoted to analyzing this population model (or very similar models, perhaps with slightly different selection mechanisms) for different ranges of values for the parameters $\mu_N$ and $s_N$.  Much of this work has been carried out by statistical physicists and appears in the biology or physics literature; see, for example, \cite{brw08, df07, f13, grbhd, gwnd, mgfd21, nh13, rbw08}.  There is also a growing body of mathematically rigorous work on the subject.  The case when $\mu_N \sim C/(N \log N)$, where one begins to see overlaps between selective sweeps, was considered by Gerrish and Lenski in \cite{gl98} and has recently been studied rigorously in \cite{ghstw}.  Durrett and Mayberry \cite{dm11} studied the case in which $s_N$ is a constant and $\mu_N \sim N^{-a}$ for some $a \in (0, 1)$.  Schweinsberg \cite{sch17a, sch17b} studied slightly faster mutation rates, so that $\mu_N$ tends to zero more slowly than any power of $1/N$.  This work made rigorous the analysis in \cite{df07, dwf13}.  Rigorous results for the case in which both $s_N$ and $\mu_N$ are constants were established in \cite{yec10, k13}.  One can also consider the case in which the mutation rate is very fast, but the selective benefit resulting from each mutation is very small.  In this case, the fitness of a lineage over time is well approximated by Brownian motion.  This parameter regime was studied by Neher and Hallatschek \cite{nh13}.  A branching Brownian motion model that should serve as a good approximation to this population model was studied rigorously in \cite{rs20, ls21}.  Finally, we note that the case when both $s_N$ and $\mu_N$ are on the scale of $1/N$ can be studied using a diffusion approximation, as discussed, for example, in section 8.1 of \cite{durr08}.

The rest of this paper is devoted to proving Theorem \ref{mainth} and Corollary \ref{maincor}.  An important component of the proof will be a coupling between the population process and a branching process with immigration, which will allow us to bound the number of individuals with a given number of mutations from above and below by branching processes.  

\section{Transition rates for the population process}

For the rest of the paper, to lighten notation, we shall omit the subscript $N$ and simply write $\mu$, $s$, $X_k(t)$, and $T_k$ in place of $\mu_N$, $s_N$, $X_{k,N}(t)$, and $T_{k,N}$.  Nevertheless, it is important to keep in mind that these quantities do depend on $N$.

In this section, we work out the transition rates for the population process.  Let
\begin{equation}\label{S}
	S(t)=\sum_{k=0}^{\infty}(1+s)^kX_k(t),
\end{equation}
which is the total fitness of the population at time $t$.
Note that $S(t)\geq \sum_{k=0}^{\infty}X_k(t)=N$ for all $t\geq 0$.  
We need to consider two types of transitions in the population process:

\begin{enumerate}
	\item For every pair of non-negative integers $(i,j)$ with $j \notin \{i, i+1\}$, $X_i$ decreases by $1$ while $X_j$ increases by 1 when a type $i$ individual is replaced by a type $j$ individual.
	Hence, the rate at which $X_i$ decreases by $1$ while $X_j$ increases by 1 at time $t$ is
	\begin{equation*}
	X_i(t)\cdot\frac{(1+s)^j X_j(t)}{S(t)}
	\end{equation*}
	because type $i$ individuals die at rate $X_i(t)$, and the probability that the new individual born is type $j$ is $(1 + s)^jX_j(t)/S(t)$.
	\item For every non-negative integer $i$, $X_i$ decreases by $1$ while $X_{i+1}$ increases by 1 when a type $i$ individual is replaced by a type $i+1$ individual, or a type $i$ individual gains a new mutation and becomes a type $i+1$ individual. 
	Hence, the rate at which $X_i$ decreases by $1$ while $X_{i+1}$ increases by 1 at time $t$ is
	\begin{equation*}
	X_i(t)\cdot\frac{(1+s)^{i+1} X_{i+1}(t)}{S(t)}+X_i(t)\mu.
	\end{equation*}
\end{enumerate}
There are also events in which a type $i$ individual is replaced by another type $i$ individual, but we may ignore these events because they do not change the composition of the population.

From these transition rates, we can see that for $k\geq 0$, the process $(X_k(t),t\geq 0)$ can be viewed as a birth-death process with immigration having the following transition rates:
\begin{enumerate}
	\item An immigration event occurs when a type $k-1$ individual becomes a type $k$ individual by acquiring a new mutation, which occurs at rate
	\begin{equation}\label{m}
		m_k(t):=X_{k-1}(t)\mu.
	\end{equation}
	Note that immigration only occurs for $k\geq 1$.  We will call a type $k$ individual a type $k$ immigrant if it arises from a type $k-1$ individual who gains a new mutation. 
	
	\item A given type $k$ individual gives birth when an individual that is not of type $k$ is replaced by a new individual who chooses this type $k$ individual as its parent.  This event occurs at rate
	\begin{equation}\label{b}
		b_k(t):=(N-X_k(t))\cdot\frac{(1+s)^k}{S(t)}.
	\end{equation}
	
	\item A given type $k$ individual dies when it is replaced by an individual that is not of type $k$, or it gains a new beneficial mutation, which occurs at rate    
	\begin{equation}\label{d}
		d_k(t):=\left(1-\frac{(1+s)^kX_k(t)}{S(t)}\right)+\mu.
	\end{equation}
\end{enumerate} 
Note that when discussing births and deaths of type $k$ individuals, we are disregarding events in which a type $k$ individual is replaced in the population by another type $k$ individual.  Ignoring these birth and death events does not affect the distribution of types in the population but does alter the genealogy of the population.  For the rest of the paper, we will work with this modified genealogy to simplify the proof.  This affects what is meant when we consider, for example, the set of individuals that are descended from a particular type $k$ immigrant.  However, because all type $k$ individuals are indistinguishable in the population, this change does not affect the distribution of types and therefore does not affect whether or not Theorem \ref{mainth} holds.

\section{Structure of the induction argument}\label{Setup}

Define $T_0'=0$, and for each positive integer $k$, let 
\begin{equation*}
T_k' = \inf\left\{t\geq T_{k-1}' : X_{k-1}(t)=0\right\}.
\end{equation*}
Note that at time $T_k'$, all individuals of types $k-1$ or lower have disappeared from the population.			
Let
\begin{equation*}
\theta=N\mu \vee \frac{1}{Ns}.
\end{equation*}
It follows from \eqref{muN} and \eqref{sN} that
\begin{equation}\label{0log}
\lim_{N \rightarrow \infty} \theta\log N = \lim_{N \rightarrow \infty} (\mu N\log N) \vee \frac{\log N}{Ns} = 0.
	\end{equation}
For positive integers $k$, let 
\begin{equation*}
\beta_k = \frac{\theta^{1/2} \mu^{k-1} (\log N)^{2k-\frac{3}{2}}}{s^k}.
\end{equation*}

\begin{lemma}\label{lim1}
We have $\beta_1\ll 1/s$, and $\beta_{k}\ll \beta_{k-1}/(\log N)^{1/2}$ for all $k\geq 2$.           
\end{lemma}

\begin{proof}
	Note that $s \beta_1 = (\theta \log N)^{1/2} \rightarrow 0$ as $N \rightarrow \infty$ by \eqref{0log}.
	Also, for positive integers $k\geq 2$, it follows from \eqref{muN} and \eqref{sN} that
	\begin{equation*}
		\frac{\beta_{k-1}}{\beta_{k}(\log N)^{1/2}}=\frac{s}{\mu (\log N)^{5/2}}\gg \frac{N^{1-\eta}}{(\log N)^{3/2}}\gg 1.
	\end{equation*}
	Thus, $\beta_{k}\ll \beta_{k-1}/(\log N)^{1/2}$ for all $k\geq 2$.           
\end{proof}

The following lemma is the key to the proof of our main results.  Note that, although for the model described in the introduction, no individuals have mutations at time zero, we present the result here under a slightly more general initial condition, so that the lemma can be applied inductively.

\begin{lemma}\label{mainlem}
	Suppose $X_k(0) \leq \beta_k$ for $1 \leq k \leq \Delta - 1$ and $X_k(0) = 0$ for $k \geq \Delta$.  Then the following hold:
	\begin{enumerate}
		\item For all $c > 0$, we have 
		\begin{equation*}
		\lim_{N \rightarrow \infty} P(N \mu s T_1 > c) = e^{-c}.
		\end{equation*}
		
		\item There exists a positive constant $C$ such that 
		\begin{equation*}
		\lim_{N \rightarrow \infty} P\left(0 \leq T_1' - T_1 < \frac{C \log N}{s} \right) = 1.
		\end{equation*}
		
		\item We have 
		\begin{equation*}
		\lim_{N \rightarrow \infty} P(T_1' < T_k \mbox{ for all }k \geq 2) = 1.
		\end{equation*}
		
		\item We have 
		\begin{equation*}
		\lim_{N \rightarrow \infty} P(X_{\Delta + 1}(t) = 0 \mbox{ for all }t \in [0, T_1']) = 1.
		\end{equation*}
		
		\item For all positive integers $k$ such that $2 \leq k \leq \Delta$, we have 
		\begin{equation*}
		\lim_{N \rightarrow \infty} P(X_k(T_1') \leq \beta_{k-1}) = 1.
		\end{equation*}
	\end{enumerate}
\end{lemma}

Part 1 of Lemma \ref{mainlem} shows that the number of type $1$ individuals reaches $(\log N)/s$ after a time which is approximately exponentially distributed with rate $N \mu s$.  Then part 2 of the lemma shows that the type 0 individuals completely disappear a short time later.  Parts 3 and 4 show that type 0 individuals disappear before the number of type $k$ individuals reaches $(\log N)/s$ for any $k \geq 2$, and before any individual acquires $\Delta + 1$ mutations.  Finally, part 5 of the lemma shows that at the time the type $0$ individuals disappear, there are at most $\beta_{k-1}$ individuals of type $k$ for $2 \leq k \leq \Delta$.

Sections \ref{Beginproof}, \ref{proofsec2}, and \ref{Endproof} are devoted to the proof of Lemma \ref{mainlem}.  In the rest of this section, we will show how to apply Lemma \ref{mainlem} inductively to obtain Lemma \ref{indlem} below, and then use Lemma~\ref{indlem} to obtain Theorem \ref{mainth} and Corollary \ref{maincor}.  Let $\mathcal{F}_t$ be the $\sigma$-field generated by the random variables $X_k(s)$ for nonnegative integers $k$ and $s \in [0, t]$, so that $(\mathcal{F}_t, t \geq 0)$ is the natural filtration associated with the population process.  Note that this filtration implicitly depends on $N$.

\begin{lemma}\label{indlem}
	For all nonnegative integers $m$, define $G_m$ to be the event that $X_{m+k}(T_m') \leq \beta_k$ for $1 \leq k \leq \Delta - 1$ and  $X_k(T_m') = 0$ for $k \notin \{m, m+1, \dots, m+\Delta-1\}$.  Then for all positive integers $m$, the following hold:
	\begin{enumerate}
		\item For all $c > 0$, we have 
		\begin{equation*}
		\lim_{N \rightarrow \infty} \big| P\big(N \mu s (T_m - T_{m-1}') > c \,|\, \mathcal{F}_{T_{m-1}'}\big) - e^{-c} \big| \mathds{1}_{G_{m-1}} = 0 \hspace{10pt} a.s.
		\end{equation*}
		
		\item There exists a positive constant $C$ such that 
		\begin{equation*}
		\lim_{N \rightarrow \infty} P\left(0 \leq T_m' - T_m < \frac{C \log N}{s} \right) = 0.
		\end{equation*}
		
		\item We have 
		\begin{equation*}
		\lim_{N \rightarrow \infty} P(T_m' < T_{m+k} \mbox{ for all }k \geq 1) = 1.
		\end{equation*}
		
		\item We have 
		\begin{equation*}
		\lim_{N \rightarrow \infty} P(X_{m+\Delta}(t) = 0 \mbox{ for all }t \in [T_{m-1}', T_m']) = 1.
		\end{equation*}
		
		\item We have 
		\begin{equation*}
		\lim_{N \rightarrow \infty} P(G_m) = 1.
		\end{equation*}
	\end{enumerate}
\end{lemma}

\begin{proof}
	The result when $m = 1$ is equivalent to Lemma \ref{mainlem}.  Suppose $m$ is a positive integer, and the result holds up to $m$.  Then $\lim_{N \rightarrow \infty} P(G_m) = 1$, so we can work on the event $G_m$.  On the event $G_m$, every individual in the population must have at least $m$ mutations from time $T_m'$ onward.  
	For $k\geq m$ and $t\geq T_m'$, on the event $G_m$,
	\begin{align*}
	b_k(t) 
	=\frac{(1+s)^k(N-X_k(t))}{\sum_{j=m}^{\infty}(1+s)^jX_j(t)} =\frac{(1+s)^{k-m}(N-X_k(t))}{\sum_{i=0}^{\infty}(1+s)^iX_{i+m}(t)}
	\end{align*}
	and 
	\begin{equation*}
	d_k(t)=1-\frac{(1+s)^kX_k(t)}{S(t)}+\mu=1-\frac{(1+s)^{k-m}X_k(t)}{\sum_{i=0}^{\infty}(1+s)^iX_{i+m}(t)}+\mu.
	\end{equation*}
	We can see from these formulas that the rates would be unchanged if $m$ were subtracted from the type of each individual, which is a consequence of the fact that subtracting $m$ from the type of each individual multiplies the fitness of each individual by $(1 + s)^{-m}$, without changing the relative fitnesses of the individuals.  Therefore, we will shift the type of each individual down by $m$, so that type $k$ individuals are relabeled as type $k-m$.
	
	After this relabeling of the types, on the event $G_m$, the distribution of types at time $T_m'$ satisfies the same conditions as the distribution of types at time zero in Lemma \ref{mainlem}.  Therefore, we can apply the strong Markov property at time $T_m'$, and after accounting for the relabeling of types, the five conclusions in Lemma \ref{mainlem} are equivalent to the five conclusions in Lemma \ref{indlem} with $m+1$ in place of $m$.  Thus, the result holds for $m+1$, and the lemma follows by induction.
\end{proof}

\begin{proof}[Proof of Theorem \ref{mainth}]
	Fix a positive integer $K$.  It follows from part 5 of Lemma \ref{indlem} that 
	\begin{equation*}
	\lim_{N \rightarrow \infty} P(G_1\cap \dots \cap G_{K-1}) = 1.
	\end{equation*}
	Therefore, by part 1 of Lemma \ref{indlem}, if $c_1, \dots, c_K > 0$, then 
	\begin{equation*}
	\lim_{N \rightarrow \infty} P\big(N \mu s(T_k - T_{k-1}') > c_k \mbox{ for }k = 1, \dots, K\big) = \prod_{k=1}^K e^{-c_k}.
	\end{equation*}
	That is, we have $(N \mu s (T_k - T_{k-1}'))_{k=1}^K \Rightarrow (\xi_k)_{k=1}^K$.  Because part 2 of Lemma \ref{indlem} and \eqref{muN} imply that $N \mu s (T_{k-1}' - T_{k-1}) \rightarrow_p 0$ as $N \rightarrow \infty$ for $k = 1, \dots K$, the result \eqref{PPPconvergence} follows.
	
	Next, note that part 5 of Lemma \ref{indlem} implies that at time $T_k'$, with probability tending to one, all individuals have type at least $k$ and at most $k+ \Delta - 1$.  Parts 2 and 3 of Lemma \ref{indlem} imply that with probability tending to one, we have $T_k' < T_{k+1} \leq T_{k+1}' < T_{k+2} \leq \dots < T_{k + \Delta}$, so in particular before time $T_{k+1}$, the number of individuals of type $j$ is less than $(\log N)/s$ for $j \in \{k+1, \dots, k+\Delta\}$.  Part 4 of Lemma \ref{indlem} implies that with probability tending to one as $N \rightarrow \infty$, no individual of type $k + \Delta + 1$ appears before time $T_{k+1}'$.  Putting together these observations, we conclude that 
	\begin{equation*}
	\lim_{N \rightarrow \infty} P\bigg( X_k(t) \geq N - \frac{\Delta \log N}{s} \mbox{ for all }t\mbox{ such that }T_k' \leq t < T_{k+1} \bigg) = 1.
	\end{equation*}
	In view of part 2 of Lemma \ref{indlem}, the result \eqref{fastsweep} follows with $C_1 = \Delta$ and $C_2 = C$.  The result \eqref{width} also follows from this same reasoning.
\end{proof}

\begin{proof}[Proof of Corollary \ref{maincor}]
	For all $u \geq 0$, let $V_N(u) = \sup\{k: T_{k,N} \leq u\}$.  It follows from \eqref{PPPconvergence} that the finite-dimensional distributions of the processes $(V_N(t/(N \mu s)), t \geq 0)$ converge as $N \rightarrow \infty$ to the finite-dimensional distributions of a homogeneous rate one Poisson process.  Therefore, it suffices to show that for each fixed $t > 0$, we have
	\begin{equation}\label{VX}
		\left|V_N\left(\frac{t}{N \mu s} \right) - \overline{X}_N\left(\frac{t}{N \mu s} \right)\right| \rightarrow_p 0 \qquad\mbox{as }N \rightarrow \infty.
	\end{equation}
	
	By \eqref{PPPconvergence}, for any fixed $t > 0$ and any $\delta > 0$, 
	\begin{equation*}
	\limsup_{N \rightarrow \infty} P \left( \frac{t - \delta}{N \mu s} \leq T_k \leq \frac{t}{N \mu s} \: \mbox{ for some }k \right) \leq \delta.
	\end{equation*} 
	Since $(\log N)/s \ll 1/(N \mu s)$ by \eqref{muN}, it follows from part 2 of Lemma \ref{indlem} that for each fixed $t > 0$, we have
	\begin{equation}\label{TkTk'}
		\lim_{N \rightarrow \infty} P\left(T_k \leq \frac{t}{N \mu s} < T_k' \: \mbox{ for some }k\right) = 0.
	\end{equation}
	However, as long as, for all $u \in [T_k', T_{k+1})$, we have $X_{j}(u) < (\log N)/s$ for $j \in \{k+1, \dots, k+\Delta\}$ and $X_{j}(u) = 0$ for all $j < k$ and $j > k + \Delta$, an event which has probability tending to one as $N \rightarrow \infty$ by Lemma \ref{indlem}, we have
	\begin{equation}\label{meanbound}
		k \leq \overline{X}_N(u) \leq k + \frac{1}{N} \sum_{j=1}^{\Delta} j X_{k+j}(u) \leq k + \frac{\Delta(\Delta+1)}{2} \cdot \frac{\log N}{Ns}
	\end{equation}
	for all $u \in [T_k', T_{k+1})$. Because $(\log N)/(Ns) \rightarrow 0$ as $N \rightarrow \infty$ by \eqref{sN}, the result \eqref{VX} follows from \eqref{TkTk'} and \eqref{meanbound}.
\end{proof}

\section{Following the process until time $T_1$}\label{Beginproof}

We now begin to work towards the proof of Lemma \ref{mainlem}.  We will therefore assume that the initial condition satisfies $X_k(0) \leq \beta_k$ for $1 \leq k \leq \Delta - 1$ and $X_k(0) = 0$ for $k \geq \Delta$.  
In this section, we study the process between time zero and the time $T_1$ when the number of type~1 individuals reaches $(\log N)/s$.  By bounding the process $X_1$ from above and below by branching processes with immigration, we will show that $T_1$ is asymptotically exponentially distributed.  We will also bound the processes $X_k$ from above to show that the number of individuals of type $2$ or higher stays small until after time $T_1$.

\subsection{Bounding the process $X_k$ from above by a branching process}

For an interval $I \subseteq [0,\infty)$, we define $X_{k,I}(t)$ to be the number of type $k$ individuals alive at time $t$ who are type $k$ immigrants that appeared during the time interval $I$, or who descended from these type $k$ immigrants.  When $0 \in I$, descendants of type $k$ individuals that are in the population at time zero are included; recall that this matters because we are aiming to prove Lemma \ref{mainlem} under slightly more general initial conditions to facilitate the induction argument.  Recall also that when determining which individuals are descended from a particular immigrant, we are ignoring events in which a type $k$ individual is replaced by another type $k$ individual.  Recall the definitions of $m_k(t), b_k(t),$ and $d_k(t)$ in (\ref{m}), (\ref{b}), and (\ref{d}), respectively. For $t \geq 0$, define
\begin{equation}\label{m1^}
	\hat{m}_1(t)=\left(\frac{N\mu}{1-N^{\eta - 1}\log N}\right)d_1(t),
\end{equation}
and for each positive integer $k$, define  
\begin{equation}\label{bk^}
	\hat{b}_k(t)=(1+s)^kd_k(t).
\end{equation} 

\begin{lemma} \label{b^>b}
	The following statements hold.
	\begin{enumerate}
		\item For all positive integers $N$ and $k$, we have $\hat{b}_k(t)\geq b_k(t)$ for all $t\geq 0$.
		\item For sufficiently large $N$, we have $\hat{m}_1(t)\geq m_1(t)$ for all $t\in (0,T_1)$.
	\end{enumerate}
\end{lemma}

\begin{proof}
	By (\ref{b}), (\ref{d}) and (\ref{bk^}), for every $t\geq 0$ and $k\geq 1$,
	\begin{align*}
		\hat{b}_k(t)
		&\geq (1+s)^k\left(1-\frac{(1+s)^kX_k(t)}{S(t)}\right)\\
		&=\frac{(1+s)^k\left(\sum_{i=0}^{k-1}(1+s)^iX_i(t)+\sum_{j=k+1}^{\infty}(1+s)^jX_j(t)\right)}{S(t)} \\
		&\geq \frac{(1+s)^k\left(\sum_{i=0}^{k-1}X_i(t)+\sum_{j=k+1}^{\infty}X_j(t)\right)}{S(t)} \\
		&=\frac{(1+s)^k(N-X_k(t))}{S(t)}\\
		&= b_k(t),
	\end{align*}
	which gives part 1 of the lemma.  To prove part 2, note that from (\ref{d}), when $t\in (0,T_1)$,
	\begin{equation*}
		d_1(t)\geq 1-\frac{(1+s)\log N}{S(t)s}+\mu\geq 1-\frac{(1+s)\log N}{Ns}.
	\end{equation*}
	Since $s \gg N^{-\eta}$, we have $d_1(t) \geq 1-N^{\eta - 1} \log N$ for sufficiently large $N$ and $t \in (0, T_1)$. The second part of the lemma now follows from (\ref{m1^}).
\end{proof}        

For positive integers $k$, let $M_k(t)$ be equal to $X_k(0)$ plus the number of times that a type $k-1$ individual mutates to type $k$ during the time interval $(0,t]$.
Define a sequence $(\gamma_N)_{N=1}^{\infty}$ such that
$$\frac{1}{(\log N)^{1/2}} \vee (\theta \log N)^{1/4} \ll \gamma_N \ll 1.$$
For $k \geq 2$, define the stopping time
\begin{align}
	\taukmut 
	&= \inf\left\{t\geq 0 : M_k(t)> \beta_{k-1} \gamma_N \right\} \label{taukmut}
\end{align}
By Lemma \ref{lim1}, we have $\beta_{k-1} \gamma_N \gg \beta_{k-1}/(\log N)^{1/2} \gg \beta_k$ for all $k \geq 2$.  Therefore, because we are assuming that $M_k(0) = X_k(0) \leq \beta_k$, for sufficiently large $N$ we have $\tau_k > 0$.

Let $I_1 = [0, T_1]$, and for $k \geq 2$, let $I_k = [0, \tau_k)$.  For all positive integers $k$, we now construct a new process $(\barYbTk(t),t\geq 0)$ from the population process as follows.  
\begin{enumerate}
	\item Set $\barYbTk(0)=0$.
	
	\item For $k = 1$ and $t \in (0, T_1)$, the process $\barYbT$ jumps up by 1 due to immigration at rate $\hat{m}_1(t)-m_1(t)$.  For $k \geq 2$, there is no immigration.
	
	\item For all $t > 0$, the process $\barYbTk$ jumps up by 1 due to births at rate \begin{equation*}
		\barYbTk(t)\hat{b}_k(t)+\XbTk(t)(\hat{b}_k(t)-b_k(t)).
	\end{equation*}
	
	\item For all $t >0$, the process $\barYbTk$ jumps down by 1 due to deaths at rate $\barYbTk(t)d_k(t)$. 
\end{enumerate}
Lemma \ref{b^>b} implies that the prescribed transition rates are nonnegative, so this process is well-defined.  Also, once the process hits 0, it cannot jump down. Thus, $\barYbTk(t)\geq 0$ for all $t\geq 0$.

One can carry out this construction formally by defining homogeneous rate one Poisson processes $(N_i(t), t \geq 0)$, $(N_{b,k}(t), t \geq 0)$, and $(N_{d,k}(t), t \geq 0)$ which are independent of one another and of the population process, and then defining the process $\barYbTk$ to satisfy
\begin{align*}
	\barYbTk(t) &= N_i \left( \int_0^{t \wedge T_1} \hat{m}_1(s)-m_1(s) \: ds \right) \mathds{1}_{\{k = 1\}} \\
	&\qquad + N_{b,k} \left( \int_0^t \barYbTk(s)\hat{b}_k(s)+\XbTk(t)(\hat{b}_k(s)-b_k(s)) \: ds \right) \\
	&\qquad \qquad- N_{d,k} \left( \int_0^t \barYbTk(s)d_k(s) \: ds \right).
\end{align*}
For other similar constructions in this paper, we will simply specify the jump rates without explicitly introducing the Poisson processes. 

For all $t \geq 0$, we define 
\begin{equation*}
	\YbTk(t)=\XbTk(t)+\barYbTk(t).
\end{equation*}
Therefore, $\YbTk(t)\geq\XbTk(t)$ for all $t\geq 0$.  Note that $\YbTk(t)$ is a birth-death process with immigration with the following rates:
\begin{enumerate}
	\item An immigrant appears in the process $\YbT$ when an immigrant appears in $\XbT$ or $\barYbT$.  Therefore, immigrants appear in $\YbT$ between times $0$ and $T_1$ at rate 
	\begin{equation*}
		m_1(t)+(\hat{m}_1(t)-m_1(t))=\hat{m}_1(t).
	\end{equation*}
	For $k \geq 2$, a immigrant appears in the process $\YbTk$ when an immigrant appears in $\XbTk$, which occurs only during the time interval $(0, \tau_k)$ at rate $m_k(t)$.
	
	\item For all $t \geq 0$, a birth occurs in the process $\YbTk$ at rate 
	\begin{equation*}
		\XbTk(t)b_k(t)+\barYbTk(t)\hat{b}_k(t)+\XbTk(t)(\hat{b}_k(t)-b_k(t)) =\YbTk(t)\hat{b}_k(t).
	\end{equation*}
	
	\item For all $t \geq 0$, a death occurs in the process $\YbTk$ at rate
	\begin{equation*}
		\XbTk(t)d_k(t)+\barYbTk(t)d_k(t)=\YbTk(t)d_k(t).
	\end{equation*}
\end{enumerate}

We shall scale the time so that each individual after the time scaling gives birth at rate $(1+s)^k$ and dies at rate $1$. 
For all positive integers $k$ and all $t\geq 0$, define
\begin{equation}\label{lambda1}
	\lambda_k(t)=\int_0^t d_k(v) \: dv,
\end{equation}
and define $\tilYbTk(t)=\YbTk(\lambda_k^{-1}(t))$.  Then the process $(\tilYbTk(t),t\geq 0)$ is a branching process with immigration with the following rates:
\begin{enumerate}
	\item When $k = 1$, immigration occurs at time $t\in (0,\lambda_1(T_1)]$ at rate
	\begin{equation*}
		\hat{m}_1(\lambda_1^{-1}(t))\cdot (\lambda_1^{-1})'(t) =\frac{\hat{m}_1(\lambda_1^{-1}(t))}{d_1(\lambda_1^{-1}(t))} =\frac{N\mu}{1-N^{\eta - 1}\log N}.
	\end{equation*}
	When $k \geq 2$, immigration occurs at a rate which is not constant in time and depends on how the population has evolved at earlier times.
	\item Each individual produces an offspring at the rate
	\begin{equation*}
		\hat{b}_k(\lambda_k^{-1}(t))\cdot (\lambda_k^{-1})'(t)=\frac{\hat{b}_k(\lambda_k^{-1}(t))}{d_k(\lambda_k^{-1}(t))} = (1+s)^k.
	\end{equation*}
	\item Each individual dies at the rate 
	\begin{equation*}
		d_k(\lambda_k^{-1}(t))\cdot (\lambda_k^{-1})'(t)=\frac{d_k(\lambda_k^{-1}(t))}{d_k(\lambda_k^{-1}(t))}=1.
	\end{equation*}
\end{enumerate}

\subsection{An upper bound on $P(T_1\leq \frac{c}{N\mu s})$}\label{upper1}

We first record the following elementary result about branching processes, which follows from classical results on asymmetric random walks.

\begin{lemma}\label{bplem}
	Consider a continuous-time branching process started from one individual in which each individual gives birth at rate $1+s$ and dies at rate $1$.  The probability that the branching process survives forever is $s/(1+s)$, and the probability that it goes extinct is $1/(1+s)$.
\end{lemma}

Define $A_{1,X}$ and $A_{1,Y}$ to be the events that the processes $\XbT$ and $\YbT$ go extinct, respectively. 

\begin{lemma}\label{PA1Y}
	We have 
	\begin{equation*}
		\lim_{N\rightarrow\infty}P(A_{1,Y})=0.
	\end{equation*}
\end{lemma}

\begin{proof}
	On the event $A_{1,Y}$, all families of individuals at time $T_1$ must go extinct.  Since individuals in the branching process $\tilYbT$ give birth at rate $1+s$ and die at rate $1$, the extinction probability of each family is $1/(1+s)$.  Also, at time $T_1$, there are at least $(\log N)/s$ individuals in the process $\YbT$ because $X_1(T_1)=\XbT(T_1)\leq \YbT(T_1)$. Hence,
	\begin{equation*}
		P(A_{1,Y}) \leq \left(\frac{1}{1+s}\right)^\frac{\log N}{s} =\left((1+s)^{-1/s}\right)^{\log N}.
	\end{equation*}
	As $N\rightarrow \infty$, we have $s\rightarrow 0$ and $(1+s)^{-1/s}\rightarrow e^{-1}$, which completes the proof.
\end{proof}

\begin{lemma}\label{PT1>c}
	For every constant $c>0$, 
	\begin{equation*}
		\limsup_{N\rightarrow\infty}P\left(T_1\leq \frac{c}{N\mu s}\right) \leq 1-e^{-c}.
	\end{equation*}
\end{lemma}

\begin{proof}
	First, we show that 
	\begin{equation}\label{limsupPA1Yc}
		\limsup_{N\rightarrow\infty} P\left(A_{1,Y}^c\cap\left\{T_1\leq \frac{c}{N\mu s}\right\}\right)\leq 1-e^{-c}.
	\end{equation}
	For $t\geq 0$, let $\tilMYsur(t)$ be the number of immigrants in the process $\tilYbT$ that appear in the time interval $(0,t]$ and whose families do not go extinct.  In the process $\tilYbT$, immigrants appear at rate $N\mu/(1-N^{\eta - 1}\log N)$ until the time $\lambda_1(T_1)$.  The family of each immigrant has extinction probability $1/(1+s)$.  Hence, the first immigrant whose family does not go extinct appears at rate 
	\begin{equation*}
		\frac{N\mu}{1-N^{\eta - 1}\log N}\cdot \frac{s}{1+s}.
	\end{equation*} 
	Also, by (\ref{d}) and (\ref{lambda1}), 
	\begin{equation}\label{lambda1T1}
		\lambda_1\left(T_1\wedge \frac{c}{N\mu s}\right)=\int_0^{T_1\wedge \frac{c}{N\mu s}} d_1(v) \: dv \leq \int_0^{\frac{c}{N\mu s}} (1+\mu) \: dv = \frac{(1+\mu)c}{N\mu s}.
	\end{equation}
	Hence,
	\begin{align*}
P\left(\tilMYsur\Big(\lambda_1\Big(T_1\wedge\frac{c}{N\mu s}\Big)\Big)>0\right) &\leq 1-\exp \left(-\frac{N\mu}{1-N^{\eta - 1}\log N}\cdot \frac{s}{1+s}\cdot\frac{(1+\mu)c}{N\mu s}\right) \\
		&= 1-\exp \left(-\frac{(1+\mu)c}{(1-N^{\eta - 1}\log N)(1+s)}\right).
	\end{align*}
	It follows that
	\begin{equation}\label{eq1*}
		\limsup_{N\rightarrow\infty} P\left(\tilMYsur\Big(\lambda_1\Big(T_1\wedge\frac{c}{N\mu s}\Big)\Big)>0\right) \leq 1-e^{-c}.
	\end{equation}
	
	Next, let $A_0$ be the event that all families of individuals at time $0$ in the process $\tilYbT$ go extinct. By the same argument in the proof of Lemma \ref{PA1Y}, since $\tilYbT(0)=X_1(0)\leq\beta_1$, we have
	\begin{equation*}
		P(A_0)\geq \left(\frac{1}{1+s}\right)^{\beta_1}=\left((1+s)^{-1/s}\right)^{s\beta _1}.
	\end{equation*}
	By Lemma \ref{lim1}, $s\beta _1\rightarrow 0$ as $N\rightarrow \infty$. Moreover, $(1+s)^{-1/s} \rightarrow e^{-1}$ as $N\rightarrow \infty$. It follows that
	\begin{equation}\label{eq2*}
		\lim_{N\rightarrow\infty}P(A_0)=1.
	\end{equation}
	
	Now, note that $A_{1,Y}^c\cap\{T_1\leq \frac{c}{N\mu s}\}\subseteq\{\tilMYsur(\lambda_1(T_1\wedge\frac{c}{N\mu s}))>0\}\cup A_0^c$.  Therefore,
	\begin{equation*}
		P\left(A_{1,Y}^c\cap\left\{T_1\leq \frac{c}{N\mu s}\right\}\right)\leq P\left(\tilMYsur\Big(\lambda_1\Big(T_1\wedge\frac{c}{N\mu s}\Big)\Big)>0\right)+P(A_0^c).
	\end{equation*}
	We obtain the inequality (\ref{limsupPA1Yc}) by taking the $\limsup$ of both sides and using (\ref{eq1*}) and (\ref{eq2*}).  Lastly, note that 
	\begin{equation*}
		P\left(T_1\leq \frac{c}{N\mu s}\right) \leq P\left(A_{1,Y}^c\cap\left\{T_1\leq \frac{c}{N\mu s}\right\}\right) + P(A_{1,Y}).
	\end{equation*}
	Thus, the result of this lemma follows by Lemma \ref{PA1Y} and (\ref{limsupPA1Yc}).
\end{proof}

\subsection{Finite and infinite lines of descent}

In this subsection, we will use the fact that a branching process that is conditioned to go extinct is still a branching process.
Let $(Y(t),t\geq 0)$ be a branching process with $Y(0)=1$. 
Let $f(x)$ be the generating function of the offspring distribution of $Y$.  
Let $b^{-1}$ be the mean lifetime of an individual in the process $Y$. 
We define $u(x)=b(f(x)-x)$. 

An individual in the branching process $Y$ is said to have a finite line of descent if the family of this particular individual goes extinct; otherwise, it is said to have an infinite line of descent.  Let $Y^{(F)}(t)$ be the number of individuals at time $t$ that have a finite line of descent, and let $Y^{(I)}(t)$ be the number of individuals at time $t$ that have an infinite line of descent. Gadag and Rajarshi \cite{gr92} showed that $((Y^{(F)}(t),Y^{(I)}(t)),t\geq 0)$ is a two-type Markov branching process. 
Let $f^{(F)}(x,y)=\sum_{i=0}^\infty\sum_{j=0}^\infty p_{ij}^{(F)}x^iy^j$ where $p_{ij}^{(F)}$ is the probability that an individual with a finite line of descent has $i$ offspring with a finite line of descent and $j$ offspring with an infinite line of descent.
Let $f^{(I)}(x,y)=\sum_{i=0}^\infty\sum_{j=0}^\infty p_{ij}^{(I)}x^iy^j$ where $p_{ij}^{(I)}$ is the probability that an individual with an infinite line of descent has $i$ offspring with a finite line of descent and $j$ offspring with an infinite line of descent.    
Also, define $u^{(F)}(x,y)=b(f^{(F)}(x,y)-x)$ and $u^{(I)}(x,y)=b(f^{(F)}(x,y)-y)$. 
Gadag and Rajarshi \cite{gr92} also showed that
\begin{equation}\label{uF}
	u^{(F)}(x,y)=\frac{u(qx)}{q}
\end{equation}
and
\begin{equation}\label{uI}
	u^{(I)}(x,y)=\frac{u(qx+(1-q)y)-u(qx)}{1-q}
\end{equation}
where $q$ is the extinction probability of the branching process $Y$.

We will apply the following result to immigrant families in the branching process $\tilYbT$.

\begin{lemma}\label{totprog}
	Let $(Y(t), t \geq 0)$ be a continuous-time branching process with $Y(0) = 1$ such that each individual gives birth at rate $1+s$ and dies at rate $1$.  Let $A$ be the event that the process goes extinct.  Then 
	\begin{equation*}
		E \bigg[ \int_0^{\infty} Y(t) \: dt \Big| A \bigg] = \frac{1}{s}.
	\end{equation*}
\end{lemma}

\begin{proof}
	We have 
	\begin{equation*}
		u(x)=1+(1+s)x^2-(2+s)x.
	\end{equation*}  
	By Lemma \ref{bplem}, the extinction probability is $1/(1+s)$, which can also be found by finding the smallest non-negative root of $u(x)$.  Thus, 
	\begin{equation*}
		u^{(F)}(x,y)=\frac{u(x/(1+s))}{1/(1+s)}=x^2+(1+s)-(2+s)x
	\end{equation*} 
	and
	\begin{equation}\label{uIxy}
		u^{(I)}(x,y)=\frac{u((x+sy)/(1+s))-u(x/(1+s))}{s/(1+s)}=sy^2+2xy-(2+s)y.
	\end{equation}
	The coefficients of $u^{(F)}(x,y)$ tell us that an individual with a finite line of descent gives birth at rate $1$ and dies at rate $1+s$. 
	Hence, for all $t > 0$, we have 
	\begin{equation*}
		E[Y(t)|A] = e^{(1-(1+s))t}=e^{-st}.
	\end{equation*}  
	The result follows by integrating over $t$.
\end{proof}

\begin{remark}\label{bprem}
	If instead each individual gives birth at rate $(1 + s)^k$ and dies at rate one, then we can apply Lemma \ref{totprog} with $(1 + s)^k - 1$ in place of $s$ to get
	\begin{equation*}
		E \bigg[ \int_0^{\infty} Y(t) \: dt \Big| A \bigg] = \frac{1}{(1 + s)^k - 1} \leq \frac{1}{sk}.
	\end{equation*}
\end{remark}

\begin{remark}\label{yulerem}
	Equation \eqref{uIxy} shows that an individual with an infinite line of descent gives birth to another individual with an infinite line of descent at rate $s$.  Therefore, conditional on $A^c$, the number of individuals with an infinite line of descent is a Yule process with birth rate $s$.
\end{remark}

\subsection{The number of type $k$ immigrants}

\begin{lemma}\label{Ptau2m>T1}
	For every constant $c>0$,
	\begin{equation}\label{46a} 
		\lim_{N\rightarrow\infty}P\left(M_2\left(T_1\wedge\frac{c}{N\mu s}\right)\leq\frac{\theta^{3/4}(\log N)^{1/4}}{s}\right)=1
	\end{equation}
	and
	\begin{equation}\label{46b}
		\lim_{N\rightarrow\infty}P\left(T_1\wedge \frac{c}{N\mu s}<\tau_2\right)=1.
	\end{equation}
\end{lemma}

\begin{proof}
	For $t\in[0,T_1)$, we have $X_1(t)\leq (\log N)/s$.  Because $M_2(0) \leq \beta_2$ and each type 1 individual can mutate to type 2 at rate $\mu$, it follows that \begin{equation*}
		E\left[M_2\left(T_1\wedge\frac{c}{N\mu s}\right)\right] \leq \beta_2 + \frac{\log N}{s}\cdot \mu \cdot \frac{c}{N\mu s} = \frac{\theta^{1/2} \mu (\log N)^{5/2}}{s^2} + \frac{c\log N}{Ns^2}.
	\end{equation*}
	Therefore, using Markov's inequality, and then using $\theta \geq N \mu$ to bound the first term and $\theta \geq \frac{1}{Ns}$ to bound the second term, we get
	\begin{align*}
P\left(M_2\left(T_1\wedge\frac{c}{N\mu s}\right) > \frac{\theta^{3/4}(\log N)^{1/4}}{s}\right) &\leq \frac{\mu (\log N)^{9/4}}{\theta^{1/4} s} + \frac{c(\log N)^{3/4}}{\theta^{3/4} Ns} \\
		&\leq \frac{(N \mu)^{3/4} (\log N)^{9/4}}{Ns} + c \left( \frac{(\log N)^3}{Ns} \right)^{1/4}.
	\end{align*}
	Because $N \mu \rightarrow 0$ and $(\log N)^a/(Ns) \rightarrow 0$ as $N \rightarrow \infty$ for all $a > 0$, the result \eqref{46a} follows.
	
	Using \eqref{0log}, we have $\theta^{3/4}(\log N)^{1/4}\ll\theta^{1/2} \ll \theta^{1/2} (\log N)^{1/2} \gamma_N = s \beta_1 \gamma_N$.  Because we have $\tau_2 = \inf\{t: M_2(t) \geq \beta_1 \gamma_N\}$, it follows that for sufficiently large $N$, 
	\begin{equation*}
		\left\{M_2\left(T_1\wedge\frac{c}{N\mu s}\right)\leq\frac{\theta^{3/4}(\log N)^{1/4}}{s}\right\}\subseteq\left\{\tau_2>T_1\wedge\frac{c}{N\mu s}\right\},
	\end{equation*} 
	which gives \eqref{46b}.
\end{proof}

For each positive integer $k\geq 2$, let $A_{k,X}$ and $A_{k,Y}$ be the events that the processes $\XbTk$ and $\YbTk$ go extinct, respectively.  Note that $A_{k,Y} \subseteq A_{k,X}$ since $\XbTk(t) \leq \YbTk(t)$ for all $t\geq 0$. 

\begin{lemma}\label{PAk}
	For every positive integer $k\geq 2$, 
	\begin{equation*}
		\lim_{N\rightarrow\infty}P(A_{k,Y})=1.
	\end{equation*}   
\end{lemma}

\begin{proof}
	In the process $\YbTk$, by Lemma \ref{bplem}, the family of each immigrant in the process $\YbTk$ goes extinct with probability $1/(1+s)^k$.  By the definition of $\taukmut$ in (\ref{taukmut}), there are at most $\beta_{k-1} \gamma_N$ immigrants during the time interval $[0,\taukmut)$, including the individuals in the population at time zero.  Then,
	\begin{equation*}
		P(A_{k,Y}) \geq \left(\frac{1}{(1+s)^k}\right)^{\beta_{k-1} \gamma_N} =\left[(1+s)^{1/s}\right]^{-ks\beta_{k-1}\gamma_N}.
	\end{equation*}
	The result follows because as $N \rightarrow \infty$, we have $s \beta_{k-1} \rightarrow 0$  by Lemma \ref{lim1}, $\gamma_N \rightarrow 0$ by definition, and $(1+s)^{1/s}\rightarrow e$.
\end{proof}

Let 
\begin{equation*}
	T_k^* = \inf\{t \geq 0: X_k(t) > \beta_{k-1}\}.
\end{equation*}

\begin{lemma}\label{PTk>taukm}
	For every positive integer $k\geq 2$, we have 
	\begin{equation*}
		\lim_{N\rightarrow\infty}P(\taukmut \leq T_k^*) = \lim_{N\rightarrow\infty}P(\taukmut \leq T_k) =1.
	\end{equation*}
\end{lemma}

\begin{proof}
By Lemma \ref{lim1}, we have $\beta_{k-1} \ll (\log N)/s$ for all $k \geq 2$, which means
$T_k^* < T_k$ for sufficiently large $N$.  Therefore, it suffices to show that $\lim_{N \rightarrow \infty} P(\tau_k \leq T_k^*) = 1$.
	Before time $\tau_k$, there are at most $\beta_{k-1} \gamma_N$ type $k$ immigrant families, counting the type $k$ individuals in the population at time zero, and Remark \ref{bprem} gives the expected value of the integral of the population size for an immigrant family conditioned on extinction.  By summing the result of Remark~\ref{bprem} over all immigrant families, we obtain
	\begin{equation}\label{intYtild1}
		E \bigg[ \bigg( \int_0^{\infty} \tilYbTk(t) \: dt \bigg) \mathds{1}_{A_{k,Y}} \bigg] \leq \frac{\beta_{k-1} \gamma_N}{(1+s)^k - 1}.
	\end{equation}
	On the event $T_k^*< \taukmut$, we have 
	\begin{equation*}
		\tilYbTk(\lambda_k(T_k^*)) = \YbTk(T_k^*) \geq X_{k,I_k}(T_k^*)>\beta_{k-1}.
	\end{equation*}
	Using the strong Markov property, we can now apply the result of Remark~\ref{bprem} again to the descendants of each of the $\tilYbTk(\lambda_k(T_k^*))$ type $k$ individuals in the population at time $\lambda_k(T_k^*)$.  We get
	\begin{equation}\label{intYtild2}
		E \bigg[ \int_0^{\infty} \tilYbTk(t) \: dt \Big| A_{k,Y} \cap \{T_k^* < \tau_k\} \bigg] \geq \frac{\beta_{k-1}}{(1+s)^k - 1}.
	\end{equation}
	Note that if $Z$ is a nonnegative random variable and $A$ and $B$ are events, then 
	\begin{equation*}
		P(A \cap B) = \frac{E[Z \mathds{1}_{A \cap B}]}{E[Z|A \cap B]} \leq \frac{E[Z\mathds{1}_A]}{E[Z|A \cap B]}.
	\end{equation*}
	Applying this result to \eqref{intYtild1} and \eqref{intYtild2}, we get
	\begin{equation*}
		P(A_{k,Y} \cap \{T_k^* < \tau_k\}) \leq \gamma_N.
	\end{equation*}
	Because $\gamma_N \rightarrow 0$ as $N \rightarrow \infty$, it now follows from Lemma \ref{PAk} that $\lim_{N \rightarrow \infty} P(T_k^* < \tau_k) = 0$, which implies the result. 
\end{proof}

\begin{lemma}\label{PtaukmINC}
	For every positive integer $k\geq 2$, we have 
	\begin{equation*}
		\lim_{N\rightarrow\infty}P(\tau_k<\tau_{k+1})=1.
	\end{equation*}
\end{lemma}

\begin{proof}
	Define the stopping time
	\begin{equation*}
		\zeta_{k,Y}=\inf\left\{t\geq 0 : \int_0^t \tilYbTk(u) \: du > \frac{\beta_{k-1}}{s}\right\}.
	\end{equation*}
	It follows from \eqref{intYtild1} and the fact that $(1+s)^k - 1 \geq sk$ that
	\begin{equation*}
		E \bigg[ \bigg( \int_0^{\infty} \tilYbTk(t) \: dt \bigg) \mathds{1}_{A_{k,Y}} \bigg] \leq \frac{\beta_{k-1} \gamma_N}{sk}.
	\end{equation*}
	Therefore, by Markov's inequality, 
	\begin{equation*}
		P(\zeta_{k,Y} < \infty) \leq P(A_{k,Y}^c) + \frac{\gamma_N}{k},
	\end{equation*}
	which by Lemma \ref{PAk} implies that 
	\begin{equation}\label{zetaeq}
		\lim_{N \rightarrow \infty} P(\zeta_{k,Y} = \infty) = 1.
	\end{equation}
	
	When $t\in [0,T_{k})$,
	\begin{equation*}
		d_k(t)=1-\frac{(1+s)^kX_k(t)}{S(t)}+\mu\geq 1-\frac{(1+s)^k\log N}{Ns}.
	\end{equation*}
	Because $N^{-\eta} \ll s \ll 1$, it follows that for sufficiently large $N$, we have $d_k(t)\geq 1/2$ for $t \in [0, T_k)$, and therefore $\lambda_k'(t) \geq 1/2$ for $t \in [0, T_k)$.  Therefore, for $t \in [0, T_k)$,
	\begin{equation*}
		\int_0^t \XbTk(u) \: du \leq \int_0^t \YbTk(u) \: du = \int_0^t \tilYbTk(\lambda_k(u)) \: du \leq 2 \int_0^{\lambda_k(t)} \tilYbTk(v) \: dv.
	\end{equation*}
	Therefore, if we let 
	\begin{equation*}
		\zeta_{k,X}=\inf\left\{t\geq 0 : \int_0^t \XbTk(u) \: du > \frac{2 \beta_{k-1}}{s}\right\},
	\end{equation*}
	then $\zeta_{k,X} \geq T_k$ on the event $\zeta_{k,Y} = \infty$.  It follows from \eqref{zetaeq} and Lemma \ref{PTk>taukm} that
	\begin{equation}\label{zetatauk}
		\lim_{N \rightarrow \infty} P(\tau_k \leq \zeta_{k,X}) = 1.
	\end{equation}
	
	Because each individual acquires mutations at rate $\mu$, we have
	\begin{align*}
		E[M_{k+1}(\zeta_{k,X} \wedge \tau_k)] &\leq M_{k+1}(0) + \frac{2 \mu \beta_{k-1}}{s} \leq \beta_{k+1} + \frac{2 \mu \beta_{k-1}}{s}.
	\end{align*}
	Because $\beta_{k+1} \ll \beta_k/(\log N)^{1/2} \ll \beta_k \gamma_N$ by Lemma \ref{lim1} and the definition of $\gamma_N$, and likewise $\mu \beta_{k-1}/s = \beta_k/(\log N)^2 \ll \beta_k \gamma_N$, it follows from Markov's inequality that
	\begin{equation*}
		\lim_{N \rightarrow \infty} P(M_{k+1}(\zeta_{k,X} \wedge \tau_k) > \beta_k \gamma_N) = 0,
	\end{equation*}
	and therefore
	\begin{equation}\label{taukk}
		\lim_{N \rightarrow \infty} P(\tau_{k+1} \leq \zeta_{k,X} \wedge \tau_k) = 0.
	\end{equation}
	The result follows from \eqref{zetatauk} and \eqref{taukk}.
\end{proof}

Recall that $\Delta = \lfloor 1/(1 - \eta) \rfloor + 1$ and $\theta=N\mu \vee \frac{1}{Ns}$.  By (\ref{taukmut}), 
\begin{equation*}
	\tau_{\Delta+1}=\inf\left\{t\geq 0 : M_{\Delta + 1}(t)>\frac{\theta^{1/2}\mu^{\Delta-1}(\log N)^{2\Delta-3/2} \gamma_N}{s^{\Delta}}\right\}.
\end{equation*}  
Since $\mu\ll 1/(N \log N)$ and $s \gg N^{-\eta}$, we have $\theta \ll 1$ and \begin{equation*}
	\frac{\mu^{\Delta -1}(\log N)^{2\Delta-3/2}}{s^{\Delta}}\ll \frac{(\log N)^{\Delta-1/2}}{N^{\Delta(1 - \eta)-1}}\ll 1.
\end{equation*}
Therefore, for sufficiently large $N$,
\begin{equation}\label{taudelta}
	\tau_{\Delta+1}=\inf\{t\geq 0 : M_{\Delta+1}(t)\geq 1\}.
\end{equation}
That is, $\tau_{\Delta+1}$ is the first time that an individual of type $\Delta + 1$  appears in the population.

Let 
\begin{equation*}
	T^{(1)}=T_2\wedge T_3\wedge ... \wedge T_{\Delta}\wedge \tau_{\Delta+1}.
\end{equation*}
By combining Lemmas \ref{Ptau2m>T1}, \ref{PTk>taukm}, and \ref{PtaukmINC}, we obtain the following result.

\begin{lemma}\label{PT1<...}
	We have
	\begin{equation*}
		\lim_{N\rightarrow\infty}P\left(T_1\wedge \frac{c}{N\mu s}<\tau_2\leq T^{(1)}\right)=1.
	\end{equation*}
\end{lemma}

\subsection{Bounding the process $X_1$ from below by a branching process}

Let $\alpha \in (0, 1)$, $\gamma \in (0,1)$, and $\zeta \in (0, 1)$ be constants.  Define 
\begin{equation*}
	T_{k,\alpha} = \inf\left\{t\geq 0 : X_k(t)>\alpha N\right\}.
\end{equation*}

\begin{lemma}\label{gamma}
	For sufficiently large $N$, we have 
	\begin{equation*}
		\frac{b_1(t)}{d_1(t)}\geq 1+\gamma s \quad\mbox{for all }t\in [0,T_{1,\alpha} \wedge T^{(1)}).
	\end{equation*}
\end{lemma}

\begin{proof}
	From (\ref{b}) and (\ref{d}), 
	\begin{equation*}
		\frac{b_1(t)}{d_1(t)} =\frac{(1+s)(N-X_1(t))}{(1+\mu)S(t)-(1+s)X_1(t)}.
	\end{equation*}
	By \eqref{taudelta}, if $N$ is sufficiently large, then during the time interval $[0,\tau_{\Delta+1})$, type $\Delta+1$ has never appeared in the population.  Hence, when $t\in[0,\tau_{\Delta+1})$, 
	\begin{equation*}
		S(t)=X_0(t)+(1+s)X_1(t)+...+(1+s)^{\Delta }X_{\Delta}(t).
	\end{equation*}
	Therefore,
	\begin{align*}
		&(1+\mu)S(t)-(1+s)X_1(t)\\
		&\quad=(1+\mu)[X_0(t)+(1+s)X_1(t)+...+(1+s)^{\Delta }X_{\Delta}(t)]-(1+s)X_1(t)\\
		&\quad=(1+\mu)[X_0(t)+(1+s)^2X_2(t)+...+(1+s)^{\Delta }X_{\Delta}(t)]+\mu(1+s)X_1(t).
	\end{align*}
	Because $s \rightarrow 0$ as $N \rightarrow \infty$, it is not difficult to show that for all $k$, we have $(1+s)^k\leq 1+2ks$ for sufficiently large $N$.  Therefore, when $t\in[0,T^{(1)})$,
	\begin{align*}
		&(1+\mu)S(t)-(1+s)X_1(t)\\
		&\quad \leq(1+\mu)[X_0(t)+(1+4s)X_2(t)+...+(1+2\Delta s)X_{\Delta}(t)]+\mu(1+s)X_1(t)\\
		&\quad =(1+\mu)[X_0(t)+X_1(t)+X_2(t)+...+X_{\Delta}(t)]+2s(1+\mu)\bigg(\sum_{j=2}^{\Delta}jX_j(t)\bigg) -(1-\mu s)X_1(t)\\
		&\quad =(1+\mu)N+2s(1+\mu)\bigg(\sum_{j=2}^{\Delta}jX_j(t)\bigg)-(1-\mu s)X_1(t)\\
		&\quad \leq (1+\mu)N+2s(1+\mu)\bigg(\sum_{j=2}^{\Delta}j\cdot\frac{\log N}{s}\bigg)-(1-\mu s)X_1(t)\\
		&\quad \leq (1+\mu)N+(1+\mu)(\Delta+1)^2\log N-(1-\mu s)X_1(t).
	\end{align*}
	Thus, when $t\in[0,T^{(1)})$, for sufficiently large $N$, 
	\begin{equation*}
		\frac{b_1(t)}{d_1(t)} \geq \frac{(1+s)(N-X_1(t))}{(1+\mu)(N+(\Delta+1)^2\log N)-(1-\mu s)X_1(t)}.
	\end{equation*}
	For all positive real numbers $a, b, c$ and $d$ such that $ac<b$, the function $f(x)= d(a-x)/(b-cx)$ is decreasing on the interval $[0,a]$; one can check that  $f'(x)=d(ac-b)/(b-cx)^2 <0$ for all $x\in (0,a)$.  Therefore, for sufficiently large $N$, when $t\in[0,T_{1,\alpha} \wedge T^{(1)})$,
	\begin{align*}
		\frac{b_1(t)}{d_1(t)} 
		&\geq \frac{(1+s)(N-\alpha N)}{(1+\mu)(N+(\Delta+1)^2\log N)-(1-\mu s)\alpha N}\\ &=\frac{(1+s)(1-\alpha)}{(1+\mu)\left(1+\frac{(\Delta+1)^2\log N}{N}\right)-(1-\mu s)\alpha}.
	\end{align*}
	Note that when $N\rightarrow \infty$, because $(\log N)/(Ns) \rightarrow 0$ and $\mu \ll s \ll 1$,
	\begin{align*}
		&\frac{1}{s}\left[\frac{(1+s)(1-\alpha)}{(1+\mu)\left(1+\frac{(\Delta+1)^2\log N}{N}\right)-(1-\mu s)\alpha} -1 \right] \\
		&\qquad \qquad =\frac{1}{s}\left[\frac{s-s(1+\mu)\alpha-\frac{(\Delta+1)^2\log N}{N}-\mu\left(1+\frac{(\Delta+1)^2\log N}{N}\right)}{(1+\mu)\left(1+\frac{(\Delta+1)^2\log N}{N}\right)-(1-\mu s)\alpha} \right]\\
		&\qquad \qquad =\frac{1-(1+\mu)\alpha-\frac{(\Delta+1)^2\log N}{Ns}-\frac{\mu}{s}\left(1+\frac{(\Delta+1)^2\log N}{N}\right)}{(1+\mu)\left(1+\frac{(\Delta+1)^2\log N}{N}\right)-(1-\mu s)\alpha} \\
		&\qquad \qquad \rightarrow 1.
	\end{align*}
	Therefore, for sufficiently large $N$,
	\begin{equation*}
		\frac{1}{s}\left[\frac{(1+s)(1-\alpha)}{(1+\mu)\left(1+\frac{(\Delta+1)^2\log N}{N}\right)-(1-\mu s)\alpha}-1 \right] \geq \gamma
	\end{equation*}
	and thus $b_1(t)/d_1(t) \geq 1+\gamma s$.
\end{proof}

\begin{lemma}\label{zeta}
	For sufficiently large $N$, we have $X_0(t)\geq (1-\zeta)N d_1(t)$ for all $t\in[0,T_1\wedge T^{(1)})$.
\end{lemma}

\begin{proof}
	From the definition of $d_1$ in (\ref{d}), $d_1(t)\leq 1+\mu$ for all $t\geq 0$. 
	Also, when $t\in [0,T_1\wedge T^{(1)})$, by \eqref{taudelta}, for sufficiently large $N$,
	\begin{equation*}
		X_0(t)=N-\sum_{i=1}^{\Delta}X_i(t)\geq N - \frac{\Delta\log N}{s}.
	\end{equation*}
	Hence, 
	\begin{equation*}
		\frac{X_0(t)}{d_1(t)}\geq\frac{N - \frac{\Delta\log N}{s}}{1+\mu}=\frac{N\left(1 - \frac{\Delta\log N}{Ns}\right)}{1+\mu}.
	\end{equation*}
	Because $(\log N)/(Ns) \rightarrow 0$ as $N \rightarrow \infty$ and $\mu \rightarrow 0$ as $N \rightarrow \infty$, we have $X_0(t)/d_1(t) \geq (1-\zeta)N$ for all $t\in [0,T_1\wedge T^{(1)})$ if $N$ is sufficiently large. 
\end{proof}

We now construct a new birth-death process with immigration called $Z_1$ which will bound the process $X_1$ from below.  We set $Z_1(0) = X_1(0)$.
\begin{enumerate}
	\item At time $t\in(0,T_1\wedge T^{(1)}]$, if a birth occurs in $X_1$, then 
	\begin{itemize}
		\item with probability $\frac{(1+\gamma s)Z_1(t)d_1(t)}{X_1(t)b_1(t)}$, a birth also occurs in $Z_1$, 
		\item with probability $1-\frac{(1+\gamma s)Z_1(t)d_1(t)}{X_1(t)b_1(t)}$, nothing happens in $Z_1$.
	\end{itemize}
	\item At time $t\in(0,T_1\wedge T^{(1)}]$, if a death occurs in $X_1$, then 
	\begin{itemize}
		\item with probability $\frac{Z_1(t)d_1(t)}{X_1(t)d_1(t)}$, a death also occurs in $Z_1$, 
		\item with probability $1-\frac{Z_1(t)d_1(t)}{X_1(t)d_1(t)}$, nothing happens in $Z_1$.
	\end{itemize}
	\item At time $t\in(0,T_1\wedge T^{(1)}]$, if an immigration event occurs in $X_1$, then
	\begin{itemize}
		\item with probability $\frac{(1-\zeta)N\mu d_1(t)}{X_0(t)\mu}$, an immigration event also occurs in $Z_1$, 
		\item with probability $1-\frac{(1-\zeta)N\mu d_1(t)}{X_0(t)\mu}$, nothing happens in $Z_1$.   
	\end{itemize}
	\item For times $t > T_1\wedge T^{(1)}$, the process $Z_1$ behaves independently of $X_1$, and
	\begin{itemize}
		\item a birth occurs at rate $(1+\gamma s)Z_1(t)d_1(t)$,
		\item a death occurs at rate $Z_1(t)d_1(t)$, 
		\item immigration occurs at rate $(1-\zeta)N\mu d_1(t)$.
	\end{itemize}
\end{enumerate}
From this construction, we see that $Z_1(t)\leq X_1(t)$ for all $t \in [0,T_1\wedge T^{(1)}]$.  Also, all of the probabilities in the construction are guaranteed to be in $[0,1]$ by Lemmas \ref{gamma} and \ref{zeta}.
Hence, $Z_1$ is a branching process with immigration where
\begin{itemize}
	\item each individual gives birth at rate $(1+\gamma s)d_1(t)$,
	\item each individual dies at rate $d_1(t)$, 
	\item immigration occurs at rate $(1-\zeta)N\mu d_1(t)$.
\end{itemize}
Recall the definition of the time scaling function $\lambda_1$ in (\ref{lambda1}). 
We define $\tilde{Z}_1(t)=Z_1(\lambda_1^{-1}(t))$ for all $t\geq 0$. 
Then the process $\tilde{Z}_1$ is a branching process with immigration in which each individual gives birth at rate $1+\gamma s$, each individual dies at rate $1$, and immigration occurs at rate $(1-\zeta)N\mu$.

By Lemma \ref{bplem}, the extinction probability of a family of an immigrant is $1/(1+\gamma s)$.  Thus, in the process $\tilde{Z}_1$, an immigrant whose family survives forever appears at rate 
\begin{equation*}
	(1-\zeta)N\mu\cdot \frac{\gamma s}{1+\gamma s}=\left(\frac{(1-\zeta)\gamma}{1+\gamma s}\right)N\mu s.
\end{equation*}
We define $\tau_{Z_1}$ to be the first time that an immigrant whose family survives forever appears in the process $\tilde{Z}_1$, and define
\begin{equation}\label{TZ1}
	T_{Z_1}=\inf\left\{t\geq 0 : Z_1(t)>\frac{\log N}{s}\right\}.
\end{equation}

\begin{lemma}\label{kappa}
	Let $\kappa \in (0,1)$ be a constant.  For sufficiently large $N$, we have \begin{equation*}
		P\left(\lambda_1(T_{Z_1})\leq \tau_{Z_1}+\frac{1}{\gamma s}\log\left(\frac{\log N}{\kappa s}\right)\right)> 1-\kappa.
	\end{equation*}
\end{lemma}

\begin{proof}
	Let $n(t)$ be the number of individuals in the process $\tilde{Z}_1$ at time $t+\tau_{Z_1}$ who have an infinite line of descent and descend from the first immigrant that has an infinite line of descent.  Let 
	\begin{equation*}
		\sigma=\inf\left\{t\geq 0 : n(t)>\frac{\log N}{s}\right\}.
	\end{equation*}
	Since $\lambda_1(T_{Z_1})$ is the first time the process $\tilde{Z}_1$ goes above $(\log N)/s$, we have $\lambda_1(T_{Z_1}) \leq \tau_{Z_1}+\sigma$.  Therefore,
	\begin{align*}
P\left(\lambda_1(T_{Z_1})\leq \tau_{Z_1}+\frac{1}{\gamma s}\log\left(\frac{\log N}{\kappa s}\right)\right)
		&\geq P\left(\sigma \leq \frac{1}{\gamma s}\log\left(\frac{\log N}{\kappa s}\right)\right)\nonumber\\
		&=P\left(n\left(\frac{1}{\gamma s}\log\left(\frac{\log N}{\kappa s}\right)\right)>\frac{\log N}{s}\right).
	\end{align*}
	It follows from Remark \ref{yulerem} that $(n(t), t \geq 0)$ is a Yule process in which each individual gives birth at rate $\gamma s$.
	Therefore, $n(t)$ has a geometric distribution with success probability $e^{-\gamma s t}$.  Using that $(\log N)/s \rightarrow \infty$ as $N \rightarrow \infty$, we have
	\begin{align*}
		P\left(n\left(\frac{1}{\gamma s}\log\left(\frac{\log N}{\kappa s}\right)\right)>\frac{\log N}{s}\right)
		&=\left[1-\exp\left(-\log\left(\frac{\log N}{\kappa s}\right)\right)\right]^{\left\lfloor\frac{\log N}{s}\right\rfloor} \\
		&=\left(1-\frac{\kappa s}{\log N}\right)^{\left\lfloor\frac{\log N}{s}\right\rfloor} \\
		&\rightarrow e^{-\kappa}
	\end{align*}
	as $N\rightarrow \infty$, and the result follows because $e^{-\kappa}> 1-\kappa$.
\end{proof}

\begin{lemma}\label{liminf}
	For every constant $c>0$, 
	\begin{equation*}
		\liminf_{N\rightarrow\infty} P\left(T_1\wedge T^{(1)}\leq \frac{c}{N\mu s}\right)\geq 1-e^{-c}.
	\end{equation*}
\end{lemma}

\begin{proof} 
	By the definition of $d_1(t)$ in (\ref{d}), for sufficiently large $N$ and for all $t \in [0, T_1)$, we have
	\begin{equation*}
		d_1(t)\geq 1-\frac{(1+s)X_1(t)}{S(t)}\geq 1-\frac{(1+s)\frac{\log N}{s}}{N}\geq 1-\frac{2\log N}{Ns}.
	\end{equation*}
	Thus, by (\ref{lambda1}), for sufficiently large $N$, when $t \in [0, T_1]$,
	\begin{equation*}
		\lambda_1(t)\geq \left(1-\frac{2\log N}{Ns}\right)t.
	\end{equation*}
	
	Because $Z_1(t) \leq X_1(t)$ for $t \in [0, T_1 \wedge T^{(1)}]$, we have $T_1 \wedge T^{(1)} \leq T_{Z_1}$.  Also, $\lambda_1$ is an increasing function.  Hence,
	\begin{align*}
		P\left(T_1\wedge T^{(1)}>\frac{c}{N\mu s}\right)
		&\leq P\left(\lambda_1(T_1\wedge T^{(1)})>\left(1-\frac{2\log N}{Ns}\right)\frac{c}{N\mu s}\right)\\
		&\leq P\left(\lambda_1(T_{Z_1})>\left(1-\frac{2\log N}{Ns}\right)\frac{c}{N\mu s}\right).
	\end{align*}
	It follows from Lemma \ref{kappa} that
	\begin{align}
P\left(T_1\wedge T^{(1)}\leq\frac{c}{N\mu s}\right)
		&\geq P\left(\lambda_1(T_{Z_1})\leq \left(1-\frac{2\log N}{Ns}\right)\frac{c}{N\mu s}\right)\nonumber\\
		&\geq P\left(\lambda_1(T_{Z_1})\leq \tau_{Z_1}+\frac{1}{\gamma s}\log\left(\frac{\log N}{\kappa s}\right) \leq \left(1-\frac{2\log N}{Ns}\right)\frac{c}{N\mu s}\right)\nonumber\\
		&\geq 1-\kappa - P\left(\tau_{Z_1}+\frac{1}{\gamma s}\log\left(\frac{\log N}{\kappa s}\right) > \left(1-\frac{2\log N}{Ns}\right)\frac{c}{N\mu s}\right)\nonumber\\
		&= 1-\kappa - P\left(\tau_{Z_1} > \frac{c}{N\mu s}\left[\left(1-\frac{2\log N}{Ns}\right)-\frac{N\mu}{c\gamma}\log\left(\frac{\log N}{\kappa s}\right)\right]\right)\label{liminf1}.
	\end{align}
	As $N \rightarrow \infty$, we have $(\log N)/(Ns) \rightarrow 0$, and because $(\log N)/(\kappa s) \leq N$ for sufficiently large $N$ by \eqref{sN}, we have
	\begin{equation*}
		\frac{N\mu}{c\gamma}\log\left(\frac{\log N}{\kappa s}\right) \leq \frac{N\mu}{c\gamma}\log N \rightarrow 0.
	\end{equation*}
	Thus, from (\ref{liminf1}), for sufficiently large $N$, 
	\begin{equation*}
		P\left(T_1\wedge T^{(1)}\leq\frac{c}{N\mu s}\right)\geq 1- \kappa - P\left(\tau_{Z_1} > \frac{c(1-\kappa)}{N\mu s}\right).
	\end{equation*}
	Since $\tau_{Z_1}$ has an exponential distribution with rate $\frac{(1-\zeta)\gamma}{1+\gamma s} \cdot N\mu s$, 
	\begin{align*}
		P\left(\tau_{Z_1} > \frac{c(1-\kappa)}{N\mu s}\right) &=\exp\left(-\frac{(1-\zeta)\gamma N\mu s}{1+\gamma s}\cdot \frac{c(1-\kappa)}{N\mu s}\right) \\
		&=\exp\left(-\frac{c(1-\kappa)(1-\zeta)\gamma}{1+\gamma s}\right).
	\end{align*}
	Thus, 
	\begin{equation*}
		P\left(T_1\wedge T^{(1)}\leq\frac{c}{N\mu s}\right)\geq 1-\kappa -\exp\left(-\frac{c(1-\kappa)(1-\zeta)\gamma}{1+\gamma s}\right).
	\end{equation*}
	It follows that 
	\begin{equation*}
		\liminf_{N\rightarrow\infty}P\left(T_1\wedge T^{(1)}\leq\frac{c}{N\mu s}\right)\geq 1-\kappa -e^{-c(1-\kappa)(1-\zeta)\gamma}.
	\end{equation*}
	Since this statement is true for all $\gamma,\zeta,\kappa \in (0,1)$, the result follows by taking limits as $\gamma\rightarrow 1^-$ and $\zeta,\kappa \rightarrow 0^+$.
\end{proof}

\begin{proposition}\label{NusT1}
	The following statements hold.
	\begin{enumerate}
		\item For every $c > 0$, we have 
		\begin{equation*}
			\lim_{N \rightarrow \infty} P(N \mu s T_1 > c) = e^{-c}.
		\end{equation*}
		
		\item We have 
		\begin{equation*}
			\lim_{N\rightarrow\infty}P(T_1<\tau_2\leq T^{(1)})=1.
		\end{equation*}
		
		\item We have 
		\begin{equation*}
			\lim_{N\rightarrow\infty}P\left(M_2(T_1)\leq\frac{\theta^{3/4}(\log N)^{1/4}}{s}\right)=1.
		\end{equation*}
	\end{enumerate}
\end{proposition}

\begin{proof}
	First, note that for every $c>0$, 
	\begin{equation*}
		P\left(T_1\wedge T^{(1)}\leq \frac{c}{N\mu s}\right) \leq P\left(T_1\leq \frac{c}{N\mu s}\right)+ P\left(T^{(1)}\leq T_1\wedge\frac{c}{N\mu s}\right).
	\end{equation*}
	From Lemma \ref{PT1<...}, we have $P(T^{(1)}\leq T_1\wedge\frac{c}{N\mu s}) \rightarrow 0$ as $N\rightarrow\infty$.  Therefore, by Lemma \ref{liminf},
	\begin{equation*}
		\liminf_{N\rightarrow\infty} P\left(T_1\leq \frac{c}{N\mu s}\right) \geq \liminf_{N\rightarrow\infty} P\left(T_1\wedge T^{(1)}\leq \frac{c}{N\mu s}\right) \geq 1-e^{-c}.
	\end{equation*}
	This result and Lemma \ref{PT1>c} imply the first part of this proposition.
	
	To prove the second statement, note that for every $c>0$, 
	\begin{equation}\label{1}
		P\left(T_1\wedge \frac{c}{N\mu s}<\tau_2\leq T^{(1)}\right) \leq P(T_1<\tau_2\leq T^{(1)})+P\left(T_1>\frac{c}{N\mu s}\right).
	\end{equation}
	By Lemma \ref{PT1<...}, the term on the left hand side converges to $1$ as $N\rightarrow\infty$.  Therefore, using the result of the first statement of the proposition and taking the liminf of both sides of \eqref{1}, we get
	\begin{equation*}
		1\leq \liminf_{N\rightarrow\infty}P(T_1<\tau_2\leq T^{(1)})+e^{-c}.
	\end{equation*}
	The result follows because this inequality is true for all $c>0$.
	
	To prove the last statement, we first observe that
	\begin{align*}
P\left(M_2(T_1)\leq\frac{\theta^{3/4}(\log N)^{1/4}}{s}\right) \geq P\left(M_2\left(T_1\wedge\frac{c}{N\mu s}\right)\leq\frac{\theta^{3/4}(\log N)^{1/4}}{s}\right) - P\left(T_1>\frac{c}{N\mu s}\right).
	\end{align*}
	Taking the liminf on both sides and using Lemma \ref{Ptau2m>T1} and part 1 of this proposition, we have 
	\begin{equation*}
		\liminf_{N\rightarrow\infty}P\left(M_2(T_1)\leq\frac{\theta^{3/4}(\log N)^{1/4}}{s}\right) \geq 1-e^{-c}.
	\end{equation*}  
	Since this is true for all $c>0$, we have proved the last statement of the proposition.
\end{proof}

\section{Following the process until time $T_{1,\alpha}$}\label{proofsec2}

In this section, we study the process between the time $T_1$ when the number of type 1 individuals first reaches $(\log N)/s$ and the time $T_{1,\alpha}$, when the number of type $1$ individuals reaches $\alpha N$.  Our goal is to show that the number of type 1 individuals reaches $\alpha N$ quickly after time $T_1$, and that there is not enough time for many mutations to type 2 to occur during this period.

We now construct a branching process which will bound $X_1$ from below between times $T_1$ and $T_{1,\alpha}$.  For $t\geq T_1\wedge T^{(1)}$, let 
\begin{equation*}
	X_1^{(1)}(t) = X_{1,[0, T_1\wedge T^{(1)}]}(t).
\end{equation*}
That is, $X_1^{(1)}(t)$ is the number of type 1 individuals alive at time $t$ who were alive and had type $1$ at time $T_1\wedge T^{(1)}$, or who descended from type $1$ individuals that were alive at time $T_1\wedge T^{(1)}$.  Similar to the way we constructed $Z_1$, we construct a new birth-death process $(Z_1'(t),t\geq T_1\wedge T^{(1)})$ from the process $X_1^{(1)}$ as follows:
\begin{enumerate}
	\item Set $Z_1'(T_1\wedge T^{(1)})=X_1^{(1)}(T_1\wedge T^{(1)})$.
	\item At time $t \in (T_1\wedge T^{(1)},T_{1,\alpha}\wedge T^{(1)}]$, if a birth occurs in $\XtoaN$, then 
	\begin{itemize}
		\item with probability $\frac{(1+\gamma s)Z'_1(t)d_1(t)}{\XtoaN(t)b_1(t)}$, a birth also occurs in $Z'_1$,
		\item with probability $1-\frac{(1+\gamma s)Z'_1(t)d_1(t)}{\XtoaN(t)b_1(t)}$, nothing happens in $Z'_1$.
	\end{itemize}
	\item At time $t \in (T_1\wedge T^{(1)},T_{1,\alpha}\wedge T^{(1)}]$, if a death occurs in $\XtoaN$, then 
	\begin{itemize}
		\item with probability $\frac{Z'_1(t)d_1(t)}{\XtoaN(t)d_1(t)}$, a death also occurs in $Z'_1$,
		\item with probability $1-\frac{Z'_1(t)d_1(t)}{\XtoaN(t)d_1(t)}$, nothing happens in $Z'_1$.
	\end{itemize}
	\item For times $t > T_{1,\alpha}\wedge T^{(1)}$, the process $Z'_1$ evolves independently of the population, and
	\begin{itemize}
		\item a birth occurs at rate $(1+\gamma s)Z'_1(t)d_1(t)$, 
		\item a death occurs at rate $Z'_1(t)d_1(t)$.
	\end{itemize}
\end{enumerate}
From this construction, which is well-defined in view of Lemma \ref{gamma}, the process $Z'_1$ is a birth-death process in which each individual gives birth at rate $(1+\gamma s)d_1(t)$ and each individual dies at rate $d_1(t)$.  Also,
\begin{equation}\label{ZXX}
	Z'_1(t)\leq X_1^{(1)}(t) \leq X_1(t) \quad\mbox{for all }t\in [T_1\wedge T^{(1)},T_{1,\alpha}\wedge T^{(1)}).
\end{equation}
For $t\geq 0$, let $\tilde{Z}_1'(t)=Z'_1(\lambda_1^{-1}(t+\lambda_1(T_1\wedge T^{(1)})))$.  Then the process $(\tilde{Z}_1'(t),t\geq 0)$ is a branching process in which each individual gives birth at rate $1+\gamma s$ and dies at rate 1.

We now review a standard result for continuous-time branching processes, which can be obtained, for example, from Theorem 6.1 on page 103 of \cite{harris63}.

\begin{lemma}\label{bpvar}
	Let $(Z(t), t \geq 0)$ be a continuous-time branching process with $Z(0) = 1$ such that each individual independently lives for an exponentially distributed time with mean $b^{-1}$ and is then replaced by $k$ offspring with probability $p_k$.  For $x \in [0,1]$, let 
	\begin{equation*}
		f(x) = \sum_{k=0}^{\infty} p_k x^k, \qquad u(x) = b(f(x) - x).
	\end{equation*}  
	Let $\lambda = u'(1)$.  Then $E[Z(t)] = e^{\lambda t}$, and if $\lambda \neq 0$, then
	\begin{equation*}
		\textup{Var}(Z(t)) = \bigg(\frac{u''(1) - \lambda}{\lambda} \bigg) (e^{2 \lambda t} - e^{\lambda t}).
	\end{equation*}
\end{lemma}

Let $D$ be the event that $T_1 < T^{(1)}$.  Then $P(D) \rightarrow 1$ as $N \rightarrow \infty$ by part 2 of Lemma~\ref{NusT1}, and on the event $D$, we have $\tilde{Z}_1'(0)=\lfloor (\log N)/s \rfloor+1$.  Also, let 
\begin{equation*}
	t_0 = \frac{1}{\gamma s} \log(\alpha N s).
\end{equation*}

\begin{lemma}\label{Z1'>aN}
	We have 
	\begin{equation*}
		\lim_{N\rightarrow\infty}P\left(\tilde{Z}_1'(t_0) > \alpha N \,\big|\,D\right)=1.
	\end{equation*}
\end{lemma}

\begin{proof}
	Since $(\tilde{Z}_1'(t),t\geq 0)$ is a branching process as described above, 
	\begin{equation*}
		E\left[\tilde{Z}_1'(t_0)\big| \,D \right] =\left(\left\lfloor\frac{\log N}{s}\right\rfloor+1\right)e^{\gamma s t_0} = \alpha N s \left(\left\lfloor\frac{\log N}{s}\right\rfloor+1\right).
	\end{equation*}
	Therefore, if $N$ is sufficiently large, $\tilde{Z}_1'(t_0)\leq\alpha N$ implies
	\begin{equation*}
		|\tilde{Z}_1'(t_0) - E[\tilde{Z}_1'(t_0)| D]| > \frac{1}{2} \alpha N \log N.
	\end{equation*}
	It thus follows from the conditional Chebyshev's inequality that
	\begin{equation}\label{Zcheb}
		P\left(\tilde{Z}_1'(t_0)\leq\alpha N\big|D\right) \leq \frac{4}{(\alpha N \log N)^2} \textup{Var}\left(\tilde{Z}_1'(t_0)\big|D\right).
	\end{equation}
	The generating function for this branching process, using the notation of Lemma \ref{bpvar}, satisfies
	\begin{equation*}
		u(x) = 1 + (1 + \gamma s)x^2 - (2 + \gamma s)x, \qquad u'(1) = \gamma s, \qquad u''(1) = 2(1 + \gamma s).
	\end{equation*}
	Therefore, by Lemma \ref{bpvar}, since the numbers of offspring produced by the $\lfloor (\log N)/s \rfloor + 1$ individuals at time zero are independent, we have
	\begin{equation}\label{Zvar}
		\textup{Var}\left(\tilde{Z}_1'(t_0)\big|D\right) \leq \left(\left\lfloor\frac{\log N}{s}\right\rfloor+1\right) \left( \frac{2 + \gamma s}{\gamma s} \right) \left( e^{2 \gamma s t_0} - e^{\gamma s t_0} \right).
	\end{equation}
	Note that $e^{2 \gamma s t_0} = (\alpha N s)^2$ and, if $N$ is sufficiently large, \begin{equation*}
		\left(\left\lfloor \frac{\log N}{s} \right\rfloor + 1\right)(2 + \gamma s) \leq \frac{3\log N}{s}.
	\end{equation*}
	Therefore, it follows from \eqref{Zcheb} and \eqref{Zvar} that for sufficiently large $N$, \begin{equation*}
		P\left(\tilde{Z}_1'(t_0)\leq\alpha N\big|D\right) \leq \frac{4}{(\alpha N \log N)^2} \cdot \frac{3 \log N}{\gamma s^2} (\alpha N s)^2 = \frac{12}{\gamma \log N},
	\end{equation*}
	which goes to 0 as $N\rightarrow \infty$.  The result of the lemma follows.
\end{proof}

\begin{lemma}\label{PT1'<}
	There is a positive constant $C$ such that
	\begin{equation*}
		\lim_{N\rightarrow\infty}P\left(T_{1,\alpha} \wedge T^{(1)} \leq T_1+ \frac{C \log N}{s} \right)=1.
	\end{equation*}
\end{lemma}

\begin{proof}
	By the definition of $d_1$ in (\ref{d}), when $t\in[0,T_{1,\alpha})$, 
	\begin{equation*}
		d_1(t)\geq 1-\frac{(1+s)X_1(t)}{S(t)}\geq 1-\frac{(1+s)\alpha N}{N}=1-(1+s)\alpha.
	\end{equation*}
	Let $\alpha'$ be a constant such that $\alpha < \alpha' < 1$.  
	Since $s\rightarrow 0$ as $N\rightarrow\infty$, for sufficiently large $N$ we have $d_1(t)\geq 1-\alpha'$ for all $t\in[0,T_{1,\alpha})$.
	By the definition of $\lambda_1$ in (\ref{lambda1}), when $0\leq u\leq t\leq T_{1,\alpha}$,
	\begin{equation}\label{l1t-l1s}
		\lambda_1(t)-\lambda_1(u) = \int_u^t d_1(v) \: dv \geq (1-\alpha')(t-u).
	\end{equation}
	Let $D^*$ be the event that $(T_1\wedge T^{(1)})+ \frac{t_0}{1 - \alpha'} < T_{1,\alpha} \wedge T^{(1)}$.  By (\ref{l1t-l1s}), on $D^*$ we have
	\begin{equation*}
		\lambda_1\left(T_1\wedge T^{(1)}+\frac{t_0}{1 - \alpha'} \right)\geq \lambda_1(T_1\wedge T^{(1)})+ t_0.
	\end{equation*}
	Since $\lambda_1$ is an increasing function, it follows that on $D^*$,
	\begin{equation*}
		T_{1,\alpha}\wedge T^{(1)}>T_1\wedge T^{(1)}+ \frac{t_0}{1 - \alpha'} > \lambda^{-1}\left(\lambda_1(T_1\wedge T^{(1)})+ t_0 \right).
	\end{equation*}
	Define $T_{1,\alpha}' = \inf\{t\geq 0 : Z_1'(t)>\alpha N\}$.  By \eqref{ZXX}, either $X_1$ reaches $\alpha N$ before or at the same time as $Z_1'$ does, or $Z_1'$ reaches $\alpha N$ after time $T^{(1)}$.  Therefore, we have $T_{1,\alpha}\wedge T^{(1)}\leq T_{1,\alpha}'$, which means that on $D^*$,
	\begin{equation*}
		T_{1,\alpha}'> \lambda^{-1}\left(\lambda_1(T_1\wedge T^{(1)})+t_0\right).
	\end{equation*}
	By the definition of $\tilde{Z}_1'$, it follows that the process $\tilde{Z}_1'$ does not go above $\alpha N$ until after time $t_0$.  That is, on $D^*$ we have
	\begin{equation*}
		\tilde{Z}_1' (t_0) \leq \alpha N.
	\end{equation*}
	Therefore, recalling that $D$ is the event that $T_1 < T^{(1)}$ and using Lemma \ref{Z1'>aN}, we have
	\begin{equation*}
		P(D^*|D) \leq P(\tilde{Z}_1' (t_0) \leq \alpha N|D) \rightarrow 0 \quad\mbox{as }N \rightarrow \infty.
	\end{equation*}
	Because $\lim_{N\rightarrow\infty}P(T_1<T^{(1)})=1$ by part 2 of Proposition \ref{NusT1}, it follows that $P(D^*) \rightarrow 0$ as $N \rightarrow \infty$, which means
	\begin{equation*}
		\lim_{N\rightarrow\infty}P\left(T_{1,\alpha} \wedge T^{(1)} \leq T_1+ \frac{t_0}{1 - \alpha'} \right)=1.
	\end{equation*}
	Because $\alpha < 1$ and $s \rightarrow 0$, for sufficiently large $N$ we have $t_0 \leq (\log N)/(\gamma s)$.  The result follows if we choose $C > 1/(\gamma (1-\alpha'))$.
\end{proof}

\begin{lemma}\label{Ptau2m>T1+clog/s}
	For all positive constants $C$, we have
	\begin{equation*}
		\lim_{N\rightarrow\infty}P\left(\tau_2>T_1+\frac{C\log N}{s}\right)=1.
	\end{equation*}
\end{lemma}

\begin{proof}
	The number of type 1 individuals at any time is bounded above by $N$, so the rate of mutations to type 2 is bounded above by $N\mu$. 
	It follows that the expected number of mutations to type 2 between times $T_1$ and $T_1 + (C \log N)/s$ is bounded above by $(C \mu N \log N)/s$. 
	By Markov's inequality and the fact that $\theta \geq N\mu$, 
	\begin{align*}
		P\left(M_2\left(T_1+\frac{C\log N}{s}\right) - M_2(T_1) > \frac{\theta^{3/4}(\log N)^{3/4}}{s}\right) 
		\leq \frac{C\mu N\log N}{\theta^{3/4}(\log N)^{3/4}} \leq C(\mu N\log N)^{1/4}.
	\end{align*}
	Since $\mu N\log N\ll 1$, we have
	\begin{equation}\label{4}
		\lim_{N\rightarrow\infty}P\left(M_2\left(T_1+\frac{C\log N}{s}\right) - M_2(T_1) \leq \frac{\theta^{3/4}(\log N)^{3/4}}{s}\right)=1.
	\end{equation}
	By Proposition \ref{NusT1} and (\ref{4}),  
	\begin{equation}\label{5}
		\lim_{N\rightarrow\infty}P\left(M_2\left(T_1+\frac{C\log N}{s}\right)\leq\frac{2\theta^{3/4}(\log N)^{3/4}}{s}\right)=1.
	\end{equation}
	Since $\gamma_N \gg (\theta \log N)^{1/4}$, we have $\theta^{3/4}(\log N)^{3/4}\ll \theta^{1/2}(\log N)^{1/2} \gamma_N = s \beta_1 \gamma_N$.
	Thus, this lemma follows from (\ref{5}) and the definition of $\tau_2$ in (\ref{taukmut}).
\end{proof}

\begin{proposition}\label{Prop2}
	We have
	\begin{equation}\label{T1T1tau2T1}
		\lim_{N\rightarrow\infty}P(T_1<T_{1,\alpha}<\tau_2\leq T^{(1)})=1.
	\end{equation}
	Also, there exists a positive constant $C$ such that
	\begin{equation}\label{alphalogN}
		\lim_{N\rightarrow\infty}P\left(T_{1,\alpha}-T_1\leq \frac{C \log N}{s}\right)=1.
	\end{equation}
\end{proposition}

\begin{proof}
	By Lemmas \ref{PT1'<} and \ref{Ptau2m>T1+clog/s}, we have $\lim_{N\rightarrow\infty}P(T_{1,\alpha}\wedge T^{(1)}<\tau_2)=1$.
	By part 2 of Proposition~\ref{NusT1}, we have $\lim_{N\rightarrow\infty}P(T_1 < \tau_2\leq T^{(1)})=1$.  Because $T_1 < T_{1,\alpha}$ for sufficiently large $N$ by definition, these two equations imply \eqref{T1T1tau2T1}.  Combining the result $\lim_{N\rightarrow\infty}P(T_{1,\alpha}< T^{(1)})=1$ with Lemma \ref{PT1'<} yields \eqref{alphalogN}.
\end{proof}

\section{Following the process until type 0 vanishes}\label{Endproof}    

In this section, we will prove that after the time $T_{1,\alpha}$ when the number of type 1 individuals reaches $\alpha N$, the type 0 population quickly goes extinct.  In particular, we will show that with probability tending to one as $N \rightarrow \infty$, we have the inequality $0<T_1'-T_{1,\alpha}\leq C'(\log N)/s$ for some positive constant $C'$.
We begin by showing that $T_{1,\alpha} < T_1'$ with probability going to 1 as $N\rightarrow\infty$. 

\begin{lemma}\label{T1'<T1''}
	We have $\lim_{N\rightarrow\infty}P(T_{1,\alpha}<T_1')=1$.
\end{lemma}
\begin{proof}
	For $t < T^{(1)}$, we have $X_k(t)\leq (\log N)/s$ for $k=2, 3,..., \Delta$.  Also, since $\tau_{\Delta+1}$ is the time that the first type $\Delta+1$ individual appears when $N$ is sufficiently large by \eqref{taudelta}, there are no individuals of type $\Delta +1$ or above for $t < T^{(1)}$.
	Therefore, as long as $T_{1,\alpha} < T^{(1)}$, we have for sufficiently large $N$,
	\begin{align*}
		X_0(T_{1,\alpha})
		&=N-\sum_{j=1}^{\Delta}X_j(T_{1,\alpha})\geq N- (\alpha N + 1) - (\Delta-1) \cdot \frac{\log N}{s}\\
		&=N\left(1-\alpha-\frac{1}{N} - \frac{(\Delta-1) \log N}{Ns}\right).
	\end{align*}
	Because $(\log N)/(Ns) \rightarrow 0$ as $N \rightarrow \infty$, when $N$ is large enough, $X_0(T_{1,\alpha})>0$ if $T_{1,\alpha} <T^{(1)}$.
	Because $\lim_{N \rightarrow \infty} P(T_{1,\alpha} < T^{(1)}) = 1$ by Proposition \ref{Prop2}, the result follows.
\end{proof}

We now bound the process $X_0$ from above by a branching process.  For $t\geq 0$, we define
\begin{equation}\label{b0^}
	\hat{b}_0(t)=\frac{d_0(t)}{1+s}.
\end{equation} 
We construct a new process $(\bar{W}_0(t),t\geq T_{1,\alpha} \wedge T_1')$ from the population process as follows. 
\begin{enumerate}
	\item Set $\bar{W}_0(T_{1,\alpha} \wedge T_1') = 0$.
	\item The process $\bar{W}_0$ jumps up by 1 at rate $\bar{W}_0\hat{b}_0(t)+X_0(t)(\hat{b}_0(t)-b_0(t))$ for all $t\geq T_{1,\alpha} \wedge T_1'$.  
	\item The process $\bar{W}_0$ jumps down by 1 at rate $\bar{W}_0(t)d_0(t)$ for all $t\geq T_{1,\alpha} \wedge T_1'$. 
\end{enumerate}
Once this process hits 0, it cannot jump down.  Therefore, $\bar{W}_0(t)\geq 0$ for all $t\geq T_{1,\alpha} \wedge T_1'$.  It remains to check that the rate at which the process $\bar{W}_0$ jumps up by 1 is non-negative, which follows from the lemma below.

\begin{lemma} 
	For $t\geq 0$, we have $b_0(t)\leq \hat{b}_0(t)$.
\end{lemma}

\begin{proof}
	By the definitions of $b_0(t)$ and $d_0(t)$ in (\ref{b}) and (\ref{d}), we have that for all $t\geq 0$,
	\begin{equation*}
		\hat{b}_0(t) =\frac{d_0(t)}{1+s} \geq \frac{1}{1+s}\left(1-\frac{X_0(t)}{S(t)}\right) = \frac{\sum_{j=1}^{\infty}(1+s)^{j-1}X_j(t)}{S(t)}
	\end{equation*}
	and 
	\begin{equation*}
		b_0(t)=\frac{N-X_0(t)}{S(t)}=\frac{\sum_{j=1}^{\infty}X_j(t)}{S(t)}.
	\end{equation*}
	Hence, $b_0(t)\leq \hat{b}_0(t)$ for all $t\geq 0$.
\end{proof}

Let $W_0(t)=X_0(t)+\bar{W}_0(t)$ for $t\geq T_{1,\alpha} \wedge T_1'$.  Clearly, $W_0(t)\geq X_0(t)$ for all $t\geq T_{1,\alpha} \wedge T_1'$.  Also, $W_0$ is a birth-death process in which, for $t\geq T_{1,\alpha} \wedge T_1'$, a birth occurs at rate $W_0(t)\hat{b}_0(t)$ and a death occurs at rate $W_0(t)d_0(t)$ for $t\geq T_{1,\alpha} \wedge T_1'$.  Next, for $t\geq T_{1,\alpha} \wedge T_1'$, define
\begin{equation}\label{lambda0}
	\lambda_0(t)=\int_{T_{1,\alpha} \wedge T_1'}^t d_0(v) \: dv.
\end{equation}
Then, we define $\tilde{W}_0(t)=W_0(\lambda_0^{-1}(t))$ for $t\geq 0$.  It follows that $\tilde{W}_0$ is a subcritical branching process in which each individual gives birth at rate $1/(1 + s)$ and dies at rate $1$. 

We define
\begin{equation}\label{tau}
	\tau=\inf\{t\geq T_{1,\alpha} \wedge T_1' : W_0(t)=0\}.
\end{equation}
and
\begin{equation}\label{tau'}
	\tau'=\inf\left\{t\geq T_{1,\alpha} \wedge T_1' : W_0(t)>\left(1-\frac{\alpha}{2}\right)N\right\}.
\end{equation}

\begin{lemma}\label{tau'=infty}
	We have $\lim_{N\rightarrow\infty} P(\tau'=\infty)=1$.
\end{lemma}

\begin{proof}
	Consider the branching process $\tilde{W}_0$.  Since each individual in this process gives birth at rate $1/(1 + s)$ and dies at rate $1$, we know that at any time, the next event is a birth with probability $1/(2 +s)$ and a death with probability $(1 + s)/(2 + s)$.  Therefore, if we evaluate the process $\tilde{W}_0$ at the time of each birth or death event, we obtain an asymmetric random walk.  Note that if $T_{1,\alpha}<T_1'$, then 
	\begin{equation*}
		\tilde{W}_0(0)=W_0(T_{1,\alpha})=X_0(T_{1,\alpha})\leq N-X_1(T_{1,\alpha})<(1-\alpha)N.
	\end{equation*} 
	Also, if $T_1' \leq T_{1,\alpha}$, then $\tilde{W}_0(0) = 0$.  Thus, in both cases, $\tilde{W}_0(0)\leq (1-\alpha)N$. 
	
	Given that $k\leq (1-\alpha)N$ and there are $k$ individuals of type $0$ at time $T_{1,\alpha} \wedge T_1'$, the probability that this random walk reaches $0$ before $\lfloor(1-\frac{\alpha}{2})N\rfloor+1$ is
	\begin{align*}
		1-\frac{(1+s)^k-1}{(1+s)^{\lfloor(1-\frac{\alpha}{2})N\rfloor+1}-1}
		&\geq 1-(1+s)^{k-\lfloor(1-\frac{\alpha}{2})N\rfloor-1}\\
		&\geq 1-(1 + s)^{(1-\alpha)N-\lfloor(1-\frac{\alpha}{2})N\rfloor-1}\\
		&\geq 1- (1+s)^{-N\alpha/2}.
	\end{align*}
	The bound on the right-hand side does not depend on $k$.  Therefore, given that $\tilde{W}_0(0)\leq (1-\alpha)N$, the probability that $\tilde{W}_0$ hits $0$ before $\lfloor(1-\frac{\alpha}{2})N\rfloor+1$ is bounded from below by $1- (1+s)^{-N\alpha/2}$.
	By the definitions of $\tau$ and $\tau'$ in \eqref{tau} and \eqref{tau'}, we have \begin{equation*}
		P(\tau<\tau' \,|\, \tilde{W}_0(0)\leq (1-\alpha)N)\geq 1- (1+s)^{-N \alpha/2}.
	\end{equation*}
	Therefore,
	\begin{equation}\label{6}
		P(\tau<\tau') \geq \left(1- (1+s)^{-N\alpha/2}\right)P(\tilde{W}_0(0)\leq (1-\alpha)N).
	\end{equation}
	Since $s\rightarrow 0$ as $N\rightarrow\infty$ and $Ns\alpha/2 \rightarrow\infty$ as $N\rightarrow\infty$, we have $(1+s)^{1/s}\rightarrow e$ and $(1+s)^{-N \alpha/2}=[(1+s)^{1/s}]^{-Ns\alpha/2}\rightarrow 0$ as $N\rightarrow\infty$.
	Thus, from (\ref{6}), $\lim_{N\rightarrow\infty}P(\tau<\tau')=1$.  Lastly, note that after time $\tau$, the process $W_0$ will stay at $0$ forever. Hence, $\tau<\tau'$ implies that $\tau'=\infty$.
\end{proof}

\begin{lemma}\label{T1''<T1'+}
	We have 
	\begin{equation*}
		\lim_{N\rightarrow\infty}P\left(T_1'\leq T_{1,\alpha}+\frac{4(1+s)\log N}{\alpha s}\right)=1.
	\end{equation*}
\end{lemma}

\begin{proof}
	Define $\tau$ and $\tau'$ as in \eqref{tau} and \eqref{tau'}.  Since $X_0(t)\leq W_0(t)$ for all $t\geq T_{1,\alpha}\wedge T_1'$, the process $X_0$ must reach $0$ before or at the same time the process $W_0$ does, which implies that $T'_1 \leq \tau$.  It is therefore enough to show that 
	\begin{equation}\label{T1alphaeq}
		\lim_{N\rightarrow\infty}P\left(\tau\leq (T_{1,\alpha}\wedge T_1') + \frac{4(1+s)\log N}{\alpha s}\right)=1.
	\end{equation}
	
	Consider the process $\tilde{W}_0$, which is a branching process in which each individual gives birth at rate $1/(1+s)$ and dies at rate $1$.  For all $t\geq 0$
	\begin{equation*}
		E[\tilde{W}_0(t)|\tilde{W}_0(0)]=\tilde{W}_0(0)e^{(\frac{1}{1+s}-1)t}=\tilde{W}_0(0)e^{-\frac{st}{1+s}}.
	\end{equation*}
	Since $\tilde{W}_0(0)\leq N$, we have $E[\tilde{W}_0(t)]\leq Ne^{-\frac{st}{1+s}}$. 
	Let $t_1=\frac{2(1+s)\log N}{s}$.  Then $E[\tilde{W}_0(t_1)] \leq 1/N$.  By Markov's inequality,
	\begin{equation*}
		P(\tilde{W}_0(t_1)=0)=1-P(\tilde{W}_0(t_1)\geq 1)\geq 1-\frac{1}{N}.
	\end{equation*}
	Hence, $\lim_{N\rightarrow\infty}P(\tilde{W}_0(t_1)=0)=1$.  Because $\tau$ is the first time that $W_0$ hits $0$, we have $\lim_{N\rightarrow\infty}P(\tau\leq \lambda_0^{-1}(t_1))=1$. 
	By Lemma \ref{tau'=infty}, we have $\lim_{N\rightarrow\infty}P(\tau\leq \lambda_0^{-1}(t_1)<\tau')=1$.
	
	Lastly, we will show that if $\lambda_0^{-1}(t_1)<\tau'$, then $\lambda_0^{-1}(t_1)\leq (T_{1,\alpha}\wedge T_1') +2t_1/\alpha$, which will imply \eqref{T1alphaeq}.  By the definitions of $d_0$ and $\lambda_0$,
	for all $t\geq T_{1,\alpha}\wedge T_1'$, we have 
	\begin{equation*}
		\lambda_0(t)=\int_{T_{1,\alpha}\wedge T_1'}^{t}\left(1-\frac{X_0(v)}{S(v)}+\mu\right)dv\geq \int_{T_{1,\alpha}\wedge T_1'}^{t}\left(1-\frac{X_0(v)}{N}\right) dv.
	\end{equation*}
	By the definition of $\tau'$ in (\ref{tau'}), if $T_{1,\alpha}\wedge T_1'\leq v < \tau'$, we have $X_0(v)\leq (1-\frac{\alpha}{2})N$.  Hence, when $T_{1,\alpha}\wedge T_1'\leq t \leq \tau'$,
	\begin{equation*}
		\lambda_0(t)\geq \int_{T_{1,\alpha}\wedge T_1'}^{t}\left(1-\frac{X_0(v)}{N}\right)dv\geq \frac{\alpha}{2}(t-(T_{1,\alpha}\wedge T_1')).
	\end{equation*}
	Therefore, if $\lambda_0^{-1}(t_1)<\tau'$, then
	\begin{equation*}
		t_1=\lambda_0(\lambda_0^{-1}(t_1))\geq \frac{\alpha}{2}(\lambda_0^{-1}(t_1)-(T_{1,\alpha}\wedge T_1')),
	\end{equation*}
	which implies that $\lambda_0^{-1}(t_1)\leq (T_{1,\alpha}\wedge T_1') + 2t_1/\alpha$.
\end{proof} 

\begin{proof}[Proof of Lemma \ref{mainlem}]
	Part 1 of Lemma \ref{mainlem} is part 1 of Proposition \ref{NusT1}.  Part~2 of Lemma \ref{mainlem} follows from \eqref{alphalogN} and Lemmas \ref{T1'<T1''} and \ref{T1''<T1'+}.  To prove parts~3 and 4 of Lemma~\ref{mainlem}, it suffices to show that $\lim_{N \rightarrow \infty} P(T_1' < T^{(1)}) = 1$.  This result holds because $\lim_{N \rightarrow \infty} P(\tau_2 \leq T^{(1)}) = 1$ by Lemma \ref{PT1<...} and $\lim_{N \rightarrow \infty} P(T_1' < \tau_2) = 1$ by Lemma \ref{Ptau2m>T1+clog/s} and part 2 of Lemma \ref{mainlem}.  Finally, to prove part 5 of Lemma \ref{mainlem}, it is enough to show that $\lim_{N \rightarrow \infty} P(T_1' < T_k^*) = 1$ for $2 \leq k \leq \Delta$.  However, we have already seen that $\lim_{N \rightarrow \infty} P(T_1' < \tau_2) = 1$, and Lemmas \ref{PTk>taukm} and \ref{PtaukmINC} imply that $\lim_{N \rightarrow \infty} P(\tau_2 \leq \tau_k \leq T_k^*) = 1$.
\end{proof}

\noindent {\bf Acknowledgments}

\bigskip
\noindent The authors thank Anton Wakolbinger for feedback which improved the exposition of the paper.  They also thank two referees for helpful comments.

\end{document}